\documentclass[a4paper,12pt]{article}

\usepackage{amscd}
\usepackage{amssymb,amsfonts,amsmath,amsthm}
\usepackage{verbatim}

\usepackage{color}
\usepackage{amsmath}
\usepackage{amsthm}
\usepackage{amssymb}

\usepackage{latexsym}
\usepackage{array}
\usepackage{enumerate}

\numberwithin{equation}{section}

\usepackage{amsmath,amssymb,mathrsfs,amsthm}
\usepackage[all]{xy}
\usepackage{graphicx}
\usepackage{ulem}
\usepackage{ascmac}
\usepackage{mathtools}

\newcommand{\BDC}{{\mathbf{D}}^{\mathrm{b}}}

\newcommand{\Mod}{\mathrm{Mod}}
\newcommand{\Hom}{\mathrm{Hom}}

\newcommand{\CC}{\mathbb{C}}

\newcommand{\RR}{\mathbb{R}}

\newcommand{\ZZ}{\mathbb{Z}}
\newcommand{\D}{\mathcal{D}}

\newcommand{\M}{\mathcal{M}}

\renewcommand{\(}{\left(}
\renewcommand{\)}{\right)}

\newcommand{\an}{{\rm an}}

\newcommand{\id}{{\rm id}}

\newcommand{\Sol}{{\rm Sol}}

\newcommand{\sub}{{\rm sub}}

\newcommand{\tl}[1]{\widetilde{#1}}

\newcommand{\simto}{\overset{\sim}{\longrightarrow}}

\newcommand{\op}{{\mbox{\scriptsize op}}}
\newcommand{\SD}{\mathcal{D}}

\newcommand{\SO}{\mathcal{O}}
\newcommand{\Ob}{\SO b}

\newcommand{\SF}{\mathcal{F}}

\newcommand{\SH}{\mathcal{H}}
\newcommand{\calI}{\mathcal{I}}

\newcommand{\Op}{\SO p}

\newcommand{\BDChol}{{\mathbf{D}}^{\mathrm{b}}_{\mbox{\rm \scriptsize hol}}}
\newcommand{\BDCrh}{{\mathbf{D}}^{\mathrm{b}}_{\mbox{\rm \scriptsize rh}}}

\newcommand{\Potimes}{\overset{+}{\otimes}}

\newcommand{\rhom}{{\bfR}{\mathcal{H}}om}
\newcommand{\rihom}{{\bfR}{\mathcal{I}}hom}

\newcommand{\Prihomsub}{{\bfR}{\mathcal{I}}hom^{+, \sub}}

\newcommand{\I}{{\rm I}}
\newcommand{\che}[1]{\check{#1}}
\newcommand{\var}[1]{\overline{#1}}
\newcommand{\BEC}{{\mathbf{E}}^{\mathrm{b}}}

\newcommand{\ZEC}{{\mathbf{E}}^{\mathrm{0}}}

\newcommand{\Q}{\mathbf{Q}}
\newcommand{\q}{\mathbf{q}}
\newcommand{\bfr}{\mathbf{r}}
\newcommand{\bfl}{\mathbf{l}}
\newcommand{\EE}{\mathbb{E}}
 
\newcommand{\bs}{\backslash}

\newcommand{\T}{{\mathsf{T}}}
\newcommand{\bfR}{\mathbf{R}}
\newcommand{\bfL}{\mathbf{L}}
\newcommand{\bfD}{\mathbf{D}}

\newcommand{\rmR}{{\rm R}}
\newcommand{\rmE}{{\rm E}}

\newcommand{\bfE}{\mathbf{E}}

\renewcommand{\Re}{\operatorname{Re}}

\newcommand{\ord}{\operatorname{ord}}

\newcommand{\AC}{(\mathbf{AC})}
\newcommand{\ACsub}{(\mathbf{AC})_{\rm sub}}
\newcommand{\rami}{{\rm rm}}
\newcommand{\modi}{{\rm md}}
\newcommand{\blow}{{\rm bl}}

\newtheorem{theorem}{Theorem}[section]

\newtheorem{corollary}[theorem]{Corollary}
\newtheorem{lemma}[theorem]{Lemma}

\newtheorem{proposition}[theorem]{Proposition}

\theoremstyle{definition}
\newtheorem{definition}[theorem]{Definition}
\theoremstyle{remark}

\title{Irregular Riemann--Hilbert Correspondence and Enhanced Subanalytic Sheaves}

\author{Yohei ITO\footnote{Department of Mathematics, Faculty of Science Division II, Tokyo University of Science, 1-3, Kagurazaka, Shinjuku-ku, Tokyo, 162-8601, Japan. E-mail: yitoh@rs.tus.ac.jp }}

\date{}

\begin{document}
\maketitle

\begin{abstract}
In \cite{Ito21b}, the author explained a relation between enhanced ind-sheaves and enhanced subanalytic sheaves.
In this paper, we shall define $\CC$-constructability for enhanced subanalytic sheaves 
which was announced in \cite[Rem.\:3.42]{Ito21b},
and show that there exists an equivalence of categories
between the triangulated category of $\CC$-constructible enhanced subanalytic sheaves
and the one of holonomic $\D$-modules.
%Although it may be known by experts, it is not in the literature to our knowledge.
\end{abstract}

\section{Introduction}
The original Riemann--Hilbert problem asks for the existence of a linear ordinary differential equation with regular singularities and a given monodromy on a curve.
In \cite{Del}, P.\:Deligne formulated it as a correspondence between integrable connections with regular singularities on a complex manifold $X$ with a pole on a hypersurface $Y$ and local systems on $X\setminus Y$.

In 1984, M.\:Kashiwara extended such a correspondence
as an equivalence of categories between
the triangulated category $\BDCrh(\D_X)$ of regular holonomic $\D_X$-modules on a complex manifold $X$
and the one $\BDC_{\CC\mbox{\scriptsize -}c}(\CC_X)$ of $\CC$-constructible sheaves on $X$.
\begin{theorem}[{\cite[Main Theorem]{Kas84}}]
There exists an equivalence of triangulated categories:
$$\Sol_X\colon \BDCrh(\D_X)^\op\simto \BDC_{\CC\mbox{\scriptsize -}c}(\CC_X),\
\M\mapsto \Sol_X(\M) := \rhom_{\D_X}(\M, \SO_X).$$
\end{theorem}

This equivalence of categories is called (analytic) regular Riemann--Hilbert correspondence
or Riemann--Hilbert correspondence for (analytic) regular holonomic $\D$-modules.

After the appearance of the regular Riemann--Hilbert correspondence,
A.\:Beilinson and J.\:Bernstein developed systematically
a theory of regular holonomic $\D$-modules on smooth algebraic varieties
over the complex number field $\CC$ and
obtained an algebraic version of the Riemann--Hilbert correspondence stated as follows.
Let $X$ be a smooth algebraic variety over $\CC$.
We denote by $X^\an$ the underling complex analytic manifold of $X$,
by $\BDCrh(\D_X)$ the triangulated category of regular holonomic $\D_X$-modules on $X$ and 
by $\BDC_{\CC\mbox{\scriptsize -}c}(\CC_X)$ the one of algebraic $\CC$-constructible sheaves on $X^\an$.
\begin{theorem}[{\cite{Be, Bor} and also \cite{Sai89}}]
There exists an equivalence of triangulated categories:
\[\Sol_X \colon \BDCrh(\D_X)^{\op}\simto\BDC_{\CC\mbox{\scriptsize -}c}(\CC_X),\
\M\mapsto \Sol_X(\M) := \Sol_{X^\an}(\M^\an),\]
here, $\M^\an$ is the analytification of $\M$.
\end{theorem}

On the other hand,
the problem of extending the (analytic) regular Riemann--Hilbert correspondence to cover
the case of (analytic) irregular holonomic $\D$-modules had been open for 30 years.
After a groundbreaking development in the theory of irregular meromorphic connections by 
K.S.\:Kedlaya \cite{Ked10, Ked11} and T.\:Mochizuki \cite{Mochi09, Mochi11},
A.\:D'Agnolo and M.\:Kashiwara established the Riemann--Hilbert correspondence
for (analytic) irregular holonomic $\SD$-modules in \cite{DK16} as below.
For this purpose, they introduced enhanced ind-sheaves extending the notion
of ind-sheaves introduced by M.\:Kashiwara and P.\:Schapira in \cite{KS01}. 
Let $X$ be a complex manifold.
We denote by $\BDChol(\D_X)$ the triangulated category of holonomic $\D_X$-modules
and denote by $\BEC_{\RR\mbox{\scriptsize -}c}(\I\CC_X)$ the one of $\RR$-constructible enhanced ind-sheaves.

\begin{theorem}[{\cite[Thm.\:9.5.3]{DK16}}]\label{irreg_DK}
There exists a fully faithful functor:
$$\Sol^\rmE_X\colon \BDChol(\D_X)^{\op}\hookrightarrow\BEC_{\RR\mbox{\scriptsize -}c}(\I\CC_X).$$
\end{theorem}

In 2016, T.\:Mochizuki proved that
the essential image of $\Sol_X^{\rmE}$ can be characterized  by the curve test \cite{Mochi22}.

In \cite{Ito20}, the author defined $\CC$-constructability for enhanced ind-sheaves
and proved that they are nothing but objects of its essentially image.
Namely, we obtain an equivalence of categories
between the triangulated category $\BDChol(\D_X)$
and the one $\BEC_{\CC\mbox{\scriptsize -}c}(\I\CC_X)$ of $\CC$-constructible enhanced ind-sheaves.
\begin{theorem}[{\cite[Thm.\:3.26]{Ito20}}]\label{irregRH_ana}
There exists an equivalence of triangulated categories:
\[\Sol_X^{\rmE} \colon \BDChol(\D_X)^{\op}\simto \BEC_{\CC\mbox{\scriptsize -}c}(\I\CC_X).\]
\end{theorem}

Remark that T.\:Kuwagaki introduced another approach
to the irregular Riemann--Hilbert correspondence
via irregular constructible sheaves which are defined by $\CC$-constructible sheaves
with coefficients in a finite version of the Novikov ring and special gradings
in \cite[Thm.\:8.6]{Kuwa21}.

Furthermore, 
the author proved an algebraic version of Theorem \ref{irregRH_ana} in \cite{Ito21} as below.
Let $X$ be a smooth algebraic variety
and denote by $\tl{X}$ a smooth completion of $X$.
%(i.e., $\tl{X}$ is a smooth complete algebraic variety over $\CC$
%which contains $X$ as an open subvariety and $\tl{X}\setminus X$ is a normal crossing divisor of $\tl{X}$).
The author defined algebraic $\CC$-constructibility for enhanced ind-sheaves
on a real analytic bordered space $X_\infty^\an = (X^\an, \tl{X}^\an)$
and proved that there exists an equivalence of categories between $\BDChol(\D_X)$
and the triangulated category $\BEC_{\CC\mbox{\scriptsize -}c}(\I\CC_{X_\infty})$
of algebraic $\CC$-constructible enhanced ind-sheaves on $X_\infty^\an$.

\begin{theorem}[{\cite[Thm.\:3.11]{Ito21}}]\label{irregRH_alg}
There exists an equivalence of triangulated categories:
\[\Sol_{X_\infty}^{\rmE} \colon \BDChol(\D_X)^{\op}\simto \BEC_{\CC\mbox{\scriptsize -}c}(\I\CC_{X_\infty}).\]
\end{theorem}

At the 16th Takagi lecture,
M.\:Kashiwara explained a similar result of Theorem.\:\ref{irreg_DK}
by using subanalytic sheaves instead of enhanced ind-sheaves as below.
We denote by $\BDC(\CC_{X\times\RR_\infty}^{\sub})$ the derived category of subanalytic sheaves
on an anlytic bordered space $X\times\RR_\infty$. 
\begin{theorem}[{\cite[\S 5.4]{Kas16}}]
Then there exists a fully faithful functor: %$\Sol_X^{\T}$ (see \cite[\S 5.4]{Kas16} for the definition)
%from $\BDChol(\D_X)$ to $\BDC(\CC_{X\times\RR_\infty}^{\sub})$:
$$\Sol_X^{\T}\colon \BDChol(\D_X)^{\op}\hookrightarrow\BDC(\CC^\sub_{X\times \RR_\infty}).$$
\end{theorem}

In \cite{Ito21b},
the author define enhanced subanalytic sheaves, and explain a relation between enhanced ind-sheaves and enhanced subanalytic sheaves as below.
Let $M_\infty$ be a real analytic bordered space.
We denote by $\BEC(\CC_{M_\infty}^\sub)$ the triangulated category of enhanced subanalytic sheaves on $M_\infty$
which is defined as the quotient category
$\BDC(\CC^\sub_{X\times \RR_\infty}) / \pi^{-1}\BDC(\CC^\sub_{X})$,
here $\pi\colon M_\infty\times \RR_\infty\to M_\infty$ is the standard projection.
The author defined $\RR$-constructability for enhanced subanalytic sheaves
and prove that there exists an equivalence of categories between 
the triangulated category $\BEC_{\RR\mbox{\scriptsize -}c}(\I\CC_{M_\infty})$
of $\RR$-constructible enhanced ind-sheaves on $M_\infty$
and the one $\BEC_{\RR\mbox{\scriptsize -}c}(\CC_{M_\infty}^\sub)$
of $\RR$-constructible enhanced subanalytic sheaves on $M_\infty$ in \cite{Ito21b}.
\begin{theorem}[{\cite[Thms.\:3.17, 3.22]{Ito21b}}]
There exists a fully faithful functors:
\begin{align*}
I^\rmE_{M_\infty}&\colon \BEC(\CC_{M_\infty}^\sub)\hookrightarrow \BEC(\I\CC_{M_\infty}),\\
\bfR_X^{\rmE, \sub}&\colon \BEC(\CC_{M_\infty}^\sub)\hookrightarrow \BDC(\CC^\sub_{X\times \RR_\infty})
\end{align*}
and an equivalence of triangulated categories:
$$I^\rmE_{M_\infty}\colon \BEC_{\RR\mbox{\scriptsize -}c}(\CC_{M_\infty}^\sub)\to \BEC_{\RR\mbox{\scriptsize -}c}(\I\CC_{M_\infty}).$$
\end{theorem}
Moreover, there exists a fully faithful functor $\Sol_X^{\rmE, \sub}$
from $\BDChol(\D_X)$ to $\BEC_{\RR\mbox{\scriptsize -}c}(\CC_{X}^\sub)$.
\begin{theorem}[{\cite[Thm.\:3.22]{Ito21b}}]\label{thm1.8}
There exists a fully faithful functor
$$\Sol_X^{\rmE, \sub}\colon \BDChol(\D_X)^{\op}\hookrightarrow\BEC_{\RR\mbox{\scriptsize -}c}(\CC_X^\sub)$$
and the following diagram is commutative:
\[\xymatrix@M=7pt@R=35pt@C=60pt{
{}&{}&\BDC(\CC_{X\times\RR_\infty}^\sub)\\
\BDChol(\D_X)
\ar@{^{(}->}[r]_-{\Sol_X^{\rmE, \sub}}
\ar@{^{(}->}[rru]
\ar@<0.6ex>@{}[rru]^-{\Sol_X^{\T, \sub}(\cdot)[1]}
\ar@{^{(}->}[rd]_-{\Sol_X^{\rmE}}
 & \BEC_{\RR\mbox{\scriptsize -}c}(\CC_X^\sub)\ar@{}[r]|-{\text{\large $\subset$}}
 \ar@<-1.0ex>@{->}[d]_-{I_X^\rmE}\ar@{}[d]|-\wr
  & \BEC(\CC_X^\sub)\ar@{^{(}->}[u]_-{\bfR_X^{\rmE, \sub}}
   \ar@<-1.0ex>@{^{(}->}[d]^-{I_X^\rmE}\\
{}&\BEC_{\RR\mbox{\scriptsize -}c}(\I\CC_X)\ar@{}[r]|-{\text{\large $\subset$}}
&\BEC(\I\CC_X).
}\]
\end{theorem}

In this paper,  we shall define $\CC$-constructability for enhanced subanalytic sheaves,
and show that there exists an equivalence of categories
between the triangulated category of $\CC$-constructible enhanced subanalytic sheaves
and the one of holonomic $\D$-modules.
Although it may be known by experts, it is not in the literature to our knowledge.
The main results are Theorem \ref{main1}, \ref{main2} for the analytic case,
Theorem \ref{main3}, \ref{main4} for the algebraic case.

%\section*{Acknowledgement}
%I would like to thank Dr. Tauchi of Kyushu University
%for many discussions and giving many comments.

\newpage

\section{Preliminary Notions and Results}\label{sec-2}
In this section,
we briefly recall some basic notions
and results which will be used in this paper. 

\subsection{Ind-Sheaves}
Let us briefly recall the notion of ind-sheaves.
References are made to \cite{KS01, KS06}. 

Let $\Bbbk$ be a field and $M$ a good topological space
(that is a topological space which is locally compact, 
Hausdorff and who has countable basis of open sets and finite soft dimension. 
We denote by $\Mod^c(\Bbbk_M)$ the category of sheaves 
of $\Bbbk$-vector spaces on $M$ with compact support.
An ind-sheaf of $\Bbbk$-vector spaces on $M$ is an ind-object of $\Mod^c(\Bbbk_M)$,
that is, inductive limit 
$$\displaystyle``\varinjlim_{i\in I}"\SF_i := \varinjlim_{i\in I}\Hom_{\Mod^c(\Bbbk_M)}(\ \cdot\ ,\ \SF_i)$$
of a small filtrant inductive system $\{\SF_i\}_{i\in I}$ in $\Mod^c(\Bbbk_M)$.
Let us denote by $\I\Bbbk_M$ the category of ind-sheaves of $\Bbbk$-vector spaces on $M$.
Note that it is abelian.
Note also that there exists a natural exact embedding $\iota_M : \Mod(\Bbbk_M)\to\I\Bbbk_M$.
It has an exact left adjoint $\alpha_M$, 
that has in turn an exact fully faithful
left adjoint functor $\beta_M$.
We denote by $\BDC(\I\Bbbk_M)$ the derived category of $\I\Bbbk_M$.
Note that there exist the Grothendieck six operations
$\otimes, \rihom, f^{-1}, \rmR f_\ast, f^!$ and $\rmR f_{!!}$
for a continuous map $f\colon M\to N$.
%Moreover, we set $\rhom := \alpha_M\circ\rihom$. 
Note also that these functors have many properties as similar to classical sheaves.
We shall skip the explanation of it.

\subsection{Subanalytic Sheaves}
We shall briefly recall the notion of subanalytic sheaves.
References are made to \cite[\S 6]{KS01}, \cite{Pre08}.

Let $\Bbbk$ be a field and $M$ a real analytic manifold.
We denote by $\Op^\sub_M$ the category of subanalytic open subsets of $M$.
Then, we can endow $\Op^\sub_M$ with the following Grothendieck topology:
a subset $S\subset \Ob((\Op_M^\sub)_U)$ is a covering of $U\in\Op_M^{\sub}$ 
if for any compact subset $K$ of $M$ there exists a finite subset $S_0\subset S$ of $S$
such that $\displaystyle U\cap K = \left(\bigcup_{V\in S_0}V\right)\cap K$.
We denote by $M^{\sub}$ such a site and call it the subanalytic site.

A subanalytic sheaf of $\Bbbk$-modules on $M$ is
a sheaf of $\Bbbk$-modules on the subanalytic site $M^{\sub}$.
We shall write $\Mod(\Bbbk_{M}^\sub)$ instead of $\Mod(\Bbbk_{M^\sub})$.
Note that it is abelian.
Note also that there exists the natural morphism $\rho_{M}\colon M\to M^\sub$ of sites.
Then we have a natural left exact embedding
$$\rho_{M\ast}\colon \Mod(\CC_M)\to\Mod(\CC_M^{\sub})$$ of categories.
It has an exact left adjoint $\rho_{M}^{-1}$
that has in turn an exact fully faithful
left adjoint functor $\rho_{M!}$.
We denote by $\BDC(\CC_M^{\sub})$ the derived category of $\Mod(\CC_M^{\sub})$.
Note that there exist the Grothendieck six operations 
$ \otimes,\ \rihom^\sub,\ \rmR f_\ast,\ \rmR f_{!!},\ f^{-1},\ f^! $
for a real analytic map $f\colon M\to N$.
Moreover, we set $\rhom^\sub := \rho_{M\ast}\circ\rihom^\sub$.
Note also that these functors have many properties as similar to classical sheaves.
We shall skip the explanation of it.

\subsection{Relation between Ind-sheaves and Subanalytic Sheaves 1}
Let us briefly recall the relation between ind-sheaves and subanalytic sheaves.
References are made to \cite[\S\S 6.3, 7.1]{KS01} and \cite[\S A.2]{Pre13}.

Let $\Bbbk$ be a field and $M$ a real analytic manifold.
We denote
by $\Mod_{\RR\mbox{\scriptsize -}c}^c(\Bbbk_M)$ the abelian category of $\RR$-constructible sheaves on $M$ with compact support
and
denote by $\I_{\RR\mbox{\scriptsize -}c}\Bbbk_M$ the category of ind-objects of $\Mod_{\RR\mbox{\scriptsize -}c}^c(\Bbbk_M)$.
Moreover let us denote by $\BDC_{\I\RR\mbox{\scriptsize -}c}(\I\Bbbk_M)$
the full triangulated subcategory of $\BDC(\I\Bbbk_{M})$
consisting of objects whose cohomologies are contained in $\I_{\RR\mbox{\scriptsize -}c}\Bbbk_{M}$.
Then there exists a functor $J_M\colon \I\Bbbk_M\to\Mod(\Bbbk_M^\sub)$ which is defined by
$$J_M\left(``\varinjlim_{i\in I}"\SF_i\right) := \varinjlim_{i\in I}\rho_{M\ast}\SF_i.$$
Note that for any $F\in\I\Bbbk_M$ the subanalytic sheaf $J_MF$ is given by 
$$\Gamma(U; J_MF) = \Hom_{\I\Bbbk_M}(\iota_M\Bbbk_U, F)$$
for each open subanalytic subset $U$.
Note also that the functor $J_M$ is left exact and admits a left adjoint
$$I_M\colon \Mod(\Bbbk_M^\sub)\to\I\Bbbk_M$$
which is fully faithful, exact and commutes with filtrant inductive limits.
Furthermore, we have:
\begin{theorem}[{\cite[Thm.\:6.3.5]{KS01}}, see also {\cite[A.2.1]{Pre13}}]\label{thm2.26}
There exists an equivalence of abelian categories:
\[\xymatrix@M=7pt@C=45pt{
\Mod(\Bbbk_M^\sub)\ar@<0.8ex>@{->}[r]^-{I_M}_-\sim
&
\I_{\RR\mbox{\scriptsize -}c}\Bbbk_M
\ar@<0.8ex>@{->}[l]^-{J_M}.
}\]
Furthermore, there exists an equivalence of triangulated categories:
\[\xymatrix@M=7pt@C=45pt{
\BDC(\Bbbk_M^\sub)\ar@<0.8ex>@{->}[r]^-{I_M}_-\sim
&
\BDC_{\I{\RR\mbox{\scriptsize -}c}}(\I\Bbbk_M)
\ar@<0.8ex>@{->}[l]^-{\bfR J_M}.
}\]
\end{theorem}

We will denote by $\lambda_M$ %\colon \BDC_{\I{\RR\mbox{\scriptsize -}c}}(\I\Bbbk_M)\simto \BDC(\Bbbk_M^\sub)$
the inverse functor of $I_M\colon \BDC(\Bbbk_M^\sub)\simto \BDC_{\I{\RR\mbox{\scriptsize -}c}}(\I\Bbbk_M)$.

\subsection{Bordered spaces}
We shall briefly recall a notion of bordered spaces.
See \cite[\S\:3.2]{DK16} and \cite[\S\:2.1]{DK21} for the details.

A bordered space is a pair $M_{\infty} = (M, \che{M})$ of
a good topological space $\che{M}$ 
and an open subset $M$ of $\che{M}$.
A morphism $f \colon (M, \che{M})\to (N, \che{N})$ of bordered spaces
is a continuous map $f \colon M\to N$ such that the first projection
$\che{M}\times\che{N}\to\che{M}$ is proper on
the closure $\var{\Gamma}_f$ of the graph $\Gamma_f$ of $f$ 
in $\che{M}\times\che{N}$. 
The category of good topological spaces is embedded into that
of bordered spaces by the identification $M = (M, M)$. 
Note that we have the morphism $j_{M_\infty} \colon M_\infty\to \che{M}$
of bordered spaces given by the embedding $M\hookrightarrow \che{M}$.
For a locally closed subset $Z\subset M$ of $M$,
we set $Z_\infty := (Z, \var{Z})$ where $\var{Z}$ is the closure of $Z$ in $\che{M}$
and denote by $i_{Z_\infty} \colon Z_\infty\to \var{Z}$ the morphism of bordered spaces
given by the natural embedding $Z\hookrightarrow \var{Z}$.

\subsection{Ind-Sheaves on Bordered Spaces}
Let us briefly recall the notion of ind-sheaves on bordered spaces.
References are made to \cite[\S 3]{DK16}.

Let $\Bbbk$ be a field and $M_\infty = (M, \che{M})$ a bordered space.
A quotient category
\begin{align*}
\BDC(\I\Bbbk_{M_\infty}) &:= 
\BDC(\I\Bbbk_{\che{M}})/\BDC(\I\Bbbk_{\che{M}\backslash M})
\end{align*}
is called the category of ind-sheaves on $M_{\infty}$,
where $\BDC(\I\Bbbk_{\che{M}\backslash M})$
is identified with its essentially image in $\BDC(\I\Bbbk_{\che{M}})$
by the fully faithful functor $\bfR i_{\che{M}\backslash M !!}$. 
Here $i_{\che{M}\backslash M}\colon \che{M}\backslash M\to M$ is the closed embedding.
An object of $\BDC(\I\Bbbk_{M_\infty})$ is called  an ind-sheaf on $M_{\infty}$.
The quotient functor $$\q_{M_\infty} : \BDC(\I\Bbbk_{\che{M}})\to\BDC(\I\Bbbk_{M_\infty})$$
has fully faithful left and right adjoints
$\bfl_{M_\infty},\bfr_{M_\infty}\colon
\BDC(\I\CC_{M_\infty}) \to\BDC(\I\CC_{\che{M}})$, respectively.
%has a left adjoint $\bfl_{M_\infty}$ and a right 
%adjoint $\bfr_{M_\infty}$, both fully faithful.
Then we have the following standard t-structure on $\BDC(\I\CC_{M_\infty})$
which is induced by the standard t-structure on $\BDC(\I\CC_{\che{M}})$
\begin{align*}
\bfD^{\leq 0}(\I\CC_{M_\infty}) & = \{F\in \BDC(\I\CC_{M_\infty})\ | \ 
\bfl_{M_\infty}F\in \bfD^{\leq 0}(\I\CC_{\che{M}})\},\\
\bfD^{\geq 0}(\I\CC_{M_\infty}) & = \{F\in \BDC(\I\CC_{M_\infty})\ | \ 
\bfl_{M_\infty}F\in \bfD^{\geq 0}(\I\CC_{\che{M}})\}.
\end{align*}

Note that there exists an embedding functor 
$\iota_{M_\infty}\colon\BDC(\Bbbk_{M}) \to \BDC(\I\Bbbk_{M_\infty})$
which have an exact left adjoint
$\alpha_{M_\infty}\colon \BDC(\I\Bbbk_{M_\infty})\to\BDC(\Bbbk_{M})$
that has in turn an exact fully faithful left adjoint functor
$\beta_{M_\infty}\colon\BDC(\Bbbk_{M}) \to \BDC(\I\Bbbk_{M_\infty})$.
Note also that there exist the Grothendieck six operations 
$ \otimes, \rihom, \rmR f_\ast, \rmR f_{!!}, f^{-1}, f^! $
for a morphism $f\colon M_\infty\to N_\infty$ 
of bordered spaces.
Moreover, these functors have many properties as similar to classical sheaves.
We shall skip the explanation of it.
We just recall that the functor
$j_{M_\infty}^{-1}\colon \BDC(\I\Bbbk_{\che{M}})\to\BDC(\I\Bbbk_{M_\infty})$ 
is isomorphic to the quotient functor
and the functor $\bfR j_{M_\infty!!}\colon\BDC(\I\Bbbk_{M_\infty})\to\BDC(\I\Bbbk_{\che{M}})$
(resp.\ $\bfR j_{M_\infty\ast}\colon\BDC(\I\Bbbk_{M_\infty})\to\BDC(\I\Bbbk_{\che{M}})$)
is isomorphic to the functor $\bfl_{M_\infty}$ (resp.\ $\bfr_{M_\infty}$).

It is clear that the quotient category
$$\BDC(\Bbbk_{M_\infty}) := 
\BDC(\Bbbk_{\che{M}})/\BDC(\Bbbk_{\che{M}\backslash M})$$
is equivalent to the derived category $\BDC(\Bbbk_{M})$
of the abelian category $\Mod(\Bbbk_M)$.
We sometimes write $\BDC(\Bbbk_{M_\infty})$ for $\BDC(\Bbbk_{M})$,
when considered as a full subcategory of $\BDC(\I\Bbbk_{M_\infty})$.

\subsection{Subanalytic Sheaves on Real Analytic Bordered spaces}
We shall briefly recall the notion of subanalytic sheaves on real analytic bordered spaces.
References are made to \cite[\S\S 3.4--3.7]{Kas16} and also \cite[\S\:3.1]{Ito21b}\footnote{
In \cite{Kas16},
M.\:Kashiwara introduced the notion of subanalytic sheaves on subanalytic bordered spaces.
However, in \cite{Ito21b}, the author only consider them on real analytic bordered spaces.}.

A real analytic bordered space is a bordered space $M_\infty = (M, \che{M})$
such that $\che{M}$ is a real analytic manifold and $M$ is an open subanalytic subset.
A morphism $f\colon (M, \che{M})\to (N, \che{N})$ of real analytic bordered spaces is
a morphism of  bordered spaces such that
the graph $\Gamma_f$ of $f$ is a subanalytic subset of $\che{M}\times\che{N}$.
The category of real analytic manifolds is embedded into that
of real analytic bordered spaces by the identification $M = (M, M)$. 

\newpage
Let $M_\infty = (M, \che{M})$ be a real analytic bordered space.
We denote by $\Op_{M_\infty}^{\sub}$ the category of open subsets of $M$ 
which are subanalytic in $\che{M}$.
Note that he category $\Op_{M_\infty}^{\sub}$ can be endowed with the following Grothendieck topology:
\begin{center}
a subset $S\subset \Ob\((\Op_{M_\infty}^{\sub})_U\)$ is a covering of $U\in\Ob\(\Op_{M_\infty}^{\sub}\)$ 
if for any compact subset $K$ of $\che{M}$ there exists a finite subset $S'\subset S$ of $S$
such that $\displaystyle K\cap U = K\cap\bigcup_{V\in S'}V.$
\end{center}
See \cite[Prop.\:3.1]{Ito21b} for the details.

Let us denote by $M_\infty^{\sub}$ the site $\Op_{M_\infty}^{\sub}$ with the above Grothendieck topology
and denote by $\Mod(\Bbbk_{M_\infty}^{\sub})$
the category of sheaves of $\Bbbk$-vector spaces on the site $M_\infty^{\sub}$.
Note that the category $\Mod(\Bbbk_{M_\infty}^{\sub})$ is abelian.
Note also that there exists the natural morphism $\rho_{M_\infty}\colon M\to M_\infty^\sub$ of sites.
Then we have a natural left exact embedding
$$\rho_{M_\infty\ast}\colon \Mod(\Bbbk_M)\to\Mod(\Bbbk_{M_\infty}^{\sub})$$ of categories.
It has an exact left adjoint $\rho_{M_\infty}^{-1}\colon \Mod(\Bbbk_{M_\infty}^{\sub})\to \Mod(\Bbbk_M)$, 
that has in turn an exact fully faithful
left adjoint functor $\rho_{M_\infty!}\colon \Mod(\Bbbk_M)\to\Mod(\Bbbk_{M_\infty}^{\sub})$.
Note that the restriction %$\rho_{M_\infty\ast}^{\RR-c}$
of $\rho_{M_\infty\ast}$
to the category $\Mod_{\RR\mbox{\scriptsize -}c}(\Bbbk_M)$ of $\RR$-constructible sheaves is exact.

We denote by $\BDC(\Bbbk_{M_\infty}^{\sub})$ the derived category of $\Mod(\Bbbk_{M_\infty}^{\sub})$.
For a morphism $f\colon M_\infty\to N_\infty$ of real analytic bordered spaces,
we have the Grothendieck operations 
$ \otimes,\ \rihom^\sub,\ \rmR f_\ast,\ \rmR f_{!!},\ f^{-1},\ f^! $
for subanalytic sheaves on bordered spaces
and set $\rhom^\sub := \rho_{M_\infty\ast}\circ\rihom^\sub$.
Note also that these functors have many properties as similar to classical sheaves.
We shall skip the explanation of it.
See \cite[\S\:3.1]{Ito21b} for the details. 

\subsection{Relation between Ind-sheaves and Subanalytic Sheaves 2}
Let us briefly recall the relation between ind-sheaves and subanalytic sheaves on real analytic bordered spaces.
References are made to \cite[\S\:3.1]{Ito21b}.

Let $\Bbbk$ be a field and $M_\infty$ a real analytic bordered space.
We set
\begin{align*}
I_{M_\infty}&\colon\BDC(\Bbbk_{M_\infty}^\sub)\to\BDC(\I\Bbbk_{M_\infty}),
\hspace{7pt}
\SF\mapsto \q_{M_\infty} I_{\che{M}}\bfR j_{M_\infty!!}\SF,\\
\bfR J_{M_\infty}&\colon\BDC(\I\Bbbk_{M_\infty})\to\BDC(\Bbbk_{M_\infty}^\sub),
\hspace{7pt}
F\mapsto j_{M_\infty}^{-1}\bfR J_{\che{M}}\bfR j_{M_\infty\ast}F,
\end{align*}  
and denote by $\BDC_{\I{\RR\mbox{\scriptsize -}c}}(\I\Bbbk_{M_\infty})$ the full triangulated subcategory of $\BDC(\I\Bbbk_{M_\infty})$
consisting of objects such that $\bfR j_{M_\infty!!}F\in\BDC_{\I\RR\mbox{\scriptsize -}c}(\I\Bbbk_{\che{M}})$.
Then we have:

\begin{theorem}[{\cite[Prop.\:3.6]{Ito21b}}]\label{thm2.2}
\begin{itemize}
\item[\rm (1)]
A pair $(I_{M_\infty}, \bfR J_{M_\infty})$ is an adjoint pair
and there exists a canonical isomorphism $\id\simto\bfR J_{M_\infty}\circ I_{M_\infty}$,

\item[\rm (2)]
There exists an equivalence of triangulated categories:
\[\xymatrix@M=7pt@C=45pt{
\BDC(\Bbbk_{M_\infty}^\sub)\ar@<0.8ex>@{->}[r]^-{I_{M_\infty}}_-\sim
&
\BDC_{\I{\RR\mbox{\scriptsize -}c}}(\I\Bbbk_{M_\infty})
\ar@<0.8ex>@{->}[l]^-{\bfR J_{M_\infty}}.
}\]
\end{itemize}
\end{theorem}

We will denote by $\lambda_{M_\infty}$ %\colon \BDC_{\I{\RR\mbox{\scriptsize -}c}}(\I\Bbbk_{M_\infty})\simto \BDC(\Bbbk_{M_\infty}^\sub)$$
the inverse functor of $I_{M_\infty}\colon \BDC(\Bbbk_{M_\infty}^\sub)
 \simto \BDC_{\I{\RR\mbox{\scriptsize -}c}}(\I\Bbbk_{M_\infty})$.

\subsection{Enhanced Ind-Sheaves}\label{subsec2.7}
We shall briefly recall the notion of enhanced ind-sheaves on bordered spaces.
References are made to \cite{KS16-2} and \cite{DK19}.
We also refer to \cite{DK16} and \cite{KS16}
for the notion of enhanced ind-sheaves on good topological spaces.

Let $\Bbbk$ be a filed and $M_\infty = (M, \che{M})$ a bordered space.
We set $\RR_\infty := (\RR, \var{\RR})$ for 
$\var{\RR} := \RR\sqcup\{-\infty, +\infty\}$,
and let $t\in\RR$ be the affine coordinate. 
We consider the morphism of bordered spaces
$\pi\colon M_\infty \times\RR_\infty\to M_\infty$
given by the projection map $\pi\colon M\times \mathbb{R}\to M, (x,t)\mapsto x$. 
Then the triangulated category of enhanced ind-sheaves on a bordered space $M_\infty$ is defined by 
$$\BEC(\I\Bbbk_{M_\infty}) :=
\BDC(\I\Bbbk_{M_\infty \times\RR_\infty})/\pi^{-1}\BDC(\I\Bbbk_{M_\infty}).$$
An object of $\BEC(\I\Bbbk_{M_\infty})$ is called  an enhanced ind-sheaf on $M_{\infty}$.
The quotient functor
$\Q_{M_\infty} \colon \BDC(\I\Bbbk_{M_\infty\times\RR_\infty})\to\BEC(\I\Bbbk_{M_\infty})$
has fully faithful left and right adjoints
$\bfL_{M_\infty}^{\rmE},\bfR_{M_\infty}^{\rmE} \colon
\BEC(\I\Bbbk_{M_\infty}) \to\BDC(\I\Bbbk_{M_\infty\times\RR_\infty})$, respectively.
Then we have the following standard t-structure on $\BEC(\I\Bbbk_{M_\infty})$
which is induced by the standard t-structure on $\BDC(\I\Bbbk_{M_\infty\times\RR_\infty})$
\begin{align*}
\bfE^{\leq 0}(\I\Bbbk_{M_\infty}) & = \{K\in \BEC(\I\Bbbk_{M_\infty})\ | \ 
\bfL_{M_\infty}^{\rmE}K\in \bfD^{\leq 0}(\I\Bbbk_{M_\infty\times\RR_\infty})\},\\
\bfE^{\geq 0}(\I\Bbbk_{M_\infty}) & = \{K\in \BEC(\I\Bbbk_{M_\infty})\ | \ 
\bfL_{M_\infty}^{\rmE}K\in \bfD^{\geq 0}(\I\Bbbk_{M_\infty\times\RR_\infty})\}.
\end{align*}
%We denote by 
%$\SH^n \colon \BEC(\I\Bbbk_{M_\infty})\to\bfE^0(\I\Bbbk_{M_\infty})$
%the $n$-th cohomology functor, where we set 
We set $$\bfE^0(\I\Bbbk_{M_\infty}) :=
\bfE^{\leq 0}(\I\Bbbk_{M_\infty})\cap\bfE^{\geq 0}(\I\Bbbk_{M_\infty}).$$
For a morphism $f \colon M_\infty\to N_\infty$ of bordered spaces, 
we have the Grothendieck six operations 
$\Potimes,\ \Prihomsub, \bfE f^{-1},\ \bfE f_\ast,\ \bfE f^!,\ \bfE f_{!!}$
for enhanced subanalytic sheaves on bordered spaces.
Note that these functors have many properties as similar to classical sheaves.
We shall skip the explanation of it.
Moreover we have external hom functors
\begin{align*}
\rihom^\rmE&\colon
\BEC(\I\Bbbk_{M_\infty})^\op\times\BEC(\I\Bbbk_{M_\infty})\to \BDC(\I\Bbbk_{M_\infty}),\
(K_1, K_2)\mapsto \bfR\pi_\ast\rihom(\bfL_{M_\infty}^\rmE K_1, \bfL_{M_\infty}^\rmE K_2)
\\
\rhom^\rmE&\colon
\BEC(\I\Bbbk_{M_\infty})^\op\times\BEC(\I\Bbbk_{M_\infty})\to \BDC(\Bbbk_{M}),\
(K_1, K_2)\mapsto \alpha_{M_\infty}\rihom^\rmE(K_1, K_2).
\end{align*}
and functors
\begin{align*}
\pi^{-1}(\cdot)\otimes (\cdot)&\colon
\BDC(\I\Bbbk_{M_\infty})\times \BEC(\I\Bbbk_{M_\infty})\to \BEC(\I\Bbbk_{M_\infty}),\
(F, K)\mapsto \Q_{M_\infty}(\pi^{-1}F\otimes \bfL_{M_\infty}^\rmE K),\\
\rihom(\pi^{-1}(\cdot), \cdot) &\colon
\BDC(\I\Bbbk_{M_\infty})^\op\times \BEC(\I\Bbbk_{M_\infty})\to \BEC(\I\Bbbk_{M_\infty}),\
(F, K)\mapsto \Q_{M_\infty}\big(\rihom(\pi^{-1}F, \bfR_{M_\infty}^\rmE K)\big).
\end{align*}

We set $$\Bbbk_{M_\infty}^\rmE := \Q_{M_\infty}\q_{M_\infty} 
\(``\underset{a\to +\infty}{\varinjlim}"\
\iota_{\che{M}\times\var{\RR}}\Bbbk_{\{t\geq a\}}\)\in\BEC(\I\Bbbk_{M_\infty}),$$
where $\{t \geq a\}$ stands for 
$\{(x, t)\in M\times{\RR}\ |\ x\in M, t \geq a\}$. 
Then we have a natural embedding functor
$$e_{M_\infty} \colon \BDC(\I\Bbbk_{M_\infty}) \to \BEC(\I\Bbbk_{M_\infty}),
\hspace{17pt}
F\mapsto\Bbbk_{M_\infty}^\rmE\otimes\pi^{-1}F.$$
Furthermore, for a continuous function $\varphi \colon U\to \RR$ defined on an open subset $U\subset M$,
we set the exponential enhanced ind-sheaf by 
%\hypertarget{E}{\[\EE_{U|M_\infty}^\varphi \colon= 
%\CC_{M_\infty}^\rmE\Potimes
%\Q_{M_\infty}\big(\CC_{\{t+\varphi\geq0\}}\big), \]}
\[\EE_{U|M_\infty}^\varphi := 
\Bbbk_{M_\infty}^\rmE\Potimes
\Q_{M_\infty}\iota_{M_\infty\times\RR_\infty}\Bbbk_{\{t+\varphi\geq0\}}
, \]
where $\{t+\varphi\geq0\}$ stands for 
$\{(x, t)\in M\times{\RR}\ |\ x\in U, t+\varphi(x)\geq0\}$. 
%$\{(x, t)\in \che{M}\times\var{\RR}\ |\ t\in\RR, x\in U, 
%t+\varphi(x)\geq0\}$. 

\subsection{$\RR$-Constructible Enhanced Ind-Sheaves} 
Let us briefly recall the definition of the $\RR$-constructible enhanced ind-sheaves.
References are made to \cite{DK16, DK19}.

Let $\Bbbk$ be a field and $M_\infty = (M, \che{M})$ a real analytic bordered space.
\begin{definition}[{\cite[Def.\:3.1.2]{DK19}}]
We denote by $\BDC_{\RR\mbox{\scriptsize -}c}(\Bbbk_{M_\infty})$
the full subcategory of $\BDC(\Bbbk_{M_\infty})$
consisting of objects $\SF$ satisfying
$\rmR j_{M_\infty!}\SF$ is an object of $\BDC_{\RR\mbox{\scriptsize -}c}(\Bbbk_{\che{M}})$.
\end{definition}

\begin{definition}[{\cite[Def.\:3.3.1]{DK19}}]\label{def2.3}
We say that $K\in\BEC(\I\Bbbk_{M_\infty})$ is $\RR$-constructible
if for any open subset $U$ of $M$ which is subanalytic and relatively compact in $\che{M}$
there exists $\SF\in\BDC_{\RR\mbox{\scriptsize -}c}(\Bbbk_{U_\infty\times\RR_\infty})$
such that $\bfE i_{U_\infty}^{-1}K\simeq
\Bbbk_{U_\infty}^{\rmE}\Potimes \Q_{U_\infty}\iota_{U_\infty\times\RR_\infty}\SF.$
\end{definition}

We denote by $\BEC_{\RR\mbox{\scriptsize -}c}(\I\Bbbk_{M_\infty})$
the full triangulated subcategory of $\BEC(\I\Bbbk_{M_\infty})$
consisting of $\RR$-constructible enhanced ind-sheaves.
Note that the triangulated category $\BEC_{\RR\mbox{\scriptsize -}c}(\I\Bbbk_{M_\infty})$
has the standard t-structure
$(\bfE^{\leq 0}_{\RR\mbox{\scriptsize -}c}(\I\Bbbk_{M_\infty}),
\bfE^{\geq 0}_{\RR\mbox{\scriptsize -}c}(\I\Bbbk_{M_\infty}))$
which is defined by
\begin{align*}
\bfE^{\leq 0}_{\RR\mbox{\scriptsize -}c}(\I\Bbbk_{M_\infty})
&:= \bfE^{\leq 0}(\I\Bbbk_{M_\infty})\cap \BEC_{\RR\mbox{\scriptsize -}c}(\I\Bbbk_{M_\infty}),\\
\bfE^{\geq 0}_{\RR\mbox{\scriptsize -}c}(\I\Bbbk_{M_\infty})
&:= \bfE^{\geq 0}(\I\Bbbk_{M_\infty})\cap \BEC_{\RR\mbox{\scriptsize -}c}(\I\Bbbk_{M_\infty}).
\end{align*}
%which is induced by the standard t-structure on $\BEC(\I\Bbbk_{M_\infty})$.
Let us set
$\ZEC_{\RR\mbox{\scriptsize -}c}(\I\Bbbk_{M_\infty})
:= \bfE^{\leq 0}_{\RR\mbox{\scriptsize -}c}(\I\Bbbk_{M_\infty})\cap
\bfE^{\geq 0}_{\RR\mbox{\scriptsize -}c}(\I\Bbbk_{M_\infty}).$

\subsection{Enhanced Subanalytic Sheaves}\label{subsec2.10}
Let us briefly recall some basic notions of enhanced subanalytic sheaves on bordered spaces and results on those.
References are made to \cite[\S\:3.3]{Ito21b}.

Let $\Bbbk$ be afield and $M_\infty = (M, \che{M})$ a real analytic bordered space.
We set $\RR_\infty := (\RR, \var{\RR})$ for 
$\var{\RR} := \RR\sqcup\{-\infty, +\infty\}$,
and let $t\in\RR$ be the affine coordinate. 
We consider the morphism of bordered spaces
$\pi\colon M_\infty \times\RR_\infty\to M_\infty$
given by the projection map $\pi\colon M\times \mathbb{R}\to M, (x,t)\mapsto x$. 
Then the triangulated category of enhanced subanalytic sheaves on a bordered space $M_\infty$ is defined by 
$$\BEC(\Bbbk_{M_\infty}^\sub) :=
\BDC(\Bbbk_{M_\infty \times\RR_\infty}^\sub)/\pi^{-1}\BDC(\Bbbk_{M_\infty}^\sub).$$
The quotient functor
$\Q_{M_\infty}^\sub \colon \BDC(\Bbbk_{M_\infty\times\RR_\infty}^\sub)\to\BEC(\Bbbk_{M_\infty}^\sub)$
has fully faithful left and right adjoints
$\bfL_{M_\infty}^{\rmE, \sub},\bfR_{M_\infty}^{\rmE, \sub} \colon
\BEC(\Bbbk_{M_\infty}^\sub) \to\BDC(\Bbbk_{M_\infty\times\RR_\infty}^\sub)$, respectively.
Let us set
\begin{align*}
\bfE^{\leq 0}(\Bbbk_{M_\infty}^\sub) & = \{K\in \BEC(\Bbbk_{M_\infty}^\sub)\ | \ 
\bfL_{M_\infty}^{\rmE, \sub}K\in \bfD^{\leq 0}(\Bbbk_{M_\infty\times\RR_\infty}^\sub)\},\\
\bfE^{\geq 0}(\Bbbk_{M_\infty}^\sub) & = \{K\in \BEC(\Bbbk_{M_\infty}^\sub)\ | \ 
\bfL_{M_\infty}^{\rmE, \sub}K\in \bfD^{\geq 0}(\Bbbk_{M_\infty\times\RR_\infty}^\sub)\},
\end{align*}
where $\(\bfD^{\leq 0}(\Bbbk_{M_\infty\times\RR_\infty}^\sub),
\bfD^{\geq 0}(\Bbbk_{M_\infty\times\RR_\infty}^\sub)\)$
is the standard t-structure on $\BDC(\Bbbk_{M_\infty\times\RR_\infty}^\sub)$.
Note that a pair $$\(\bfE^{\leq 0}(\Bbbk_{M_\infty}^\sub),
\bfE^{\geq 0}(\Bbbk_{M_\infty}^\sub)\)$$
is a t-structure on $\BEC(\Bbbk_{M_\infty}^\sub)$.
%We denote by 
%\[\SH^n \colon \BEC(\Bbbk_{M_\infty}^\sub)\to\bfE^0(\Bbbk_{M_\infty}^\sub)\]
%the $n$-th cohomology functor, where we set 
We set $$\bfE^0(\Bbbk_{M_\infty}^\sub) :=
\bfE^{\leq 0}(\Bbbk_{M_\infty}^\sub)\cap\bfE^{\geq 0}(\Bbbk_{M_\infty}^\sub).$$
For a morphism $f \colon M_\infty\to N_\infty$ of bordered spaces, 
we have the Grothendieck six operations 
$\Potimes,\ \Prihomsub, \bfE f^{-1},\ \bfE f_\ast,\ \bfE f^!,\ \bfE f_{!!}$
for enhanced subanalytic sheaves on bordered spaces.
Note that these functors have many properties as similar to classical sheaves.
We shall skip the explanation of it.
Moreover, we have functors
{\small\begin{align*}
\pi^{-1}(\cdot)\otimes (\cdot)&\colon
\BDC(\Bbbk_{M_\infty}^\sub)\times \BEC(\Bbbk_{M_\infty}^\sub)\to \BEC(\Bbbk_{M_\infty}^\sub),\
(\SF, K)\mapsto \Q_{M_\infty}^\sub(\pi^{-1}\SF\otimes \bfL_{M_\infty}^{\rmE, \sub} K),\\
\rihom^\sub(\pi^{-1}(\cdot), \cdot) &\colon
\BDC(\Bbbk_{M_\infty}^\sub)^\op\times \BEC(\Bbbk_{M_\infty}^\sub)\to \BEC(\Bbbk_{M_\infty}^\sub),\
(\SF, K)\mapsto \Q_{M_\infty}^\sub\big(\rihom^\sub(\pi^{-1}\SF, \bfR_{M_\infty}^{\rmE, \sub} K)\big).
\end{align*}}

We set $$\Bbbk_{M_\infty}^{\rmE, \sub} := \Q_{M_\infty}^\sub 
\(\underset{a\to +\infty}{\varinjlim}\ \rho_{M_\infty\times\RR_\infty\ast}\Bbbk_{\{t\geq a\}}
\)\in\BEC(\Bbbk_{M_\infty}^\sub)$$
where $\{t\geq a\}$ stands for $\{(x, t)\in M\times\RR\ |\ t\geq a\}\subset \che{M}\times\var{\RR}$.
Then we have a natural embedding
$$e_{M_\infty}^\sub \colon \BDC(\Bbbk_{M_\infty}^\sub) \to \BEC(\Bbbk_{M_\infty}^\sub),
\hspace{7pt}
\SF\mapsto \pi^{-1}\SF\otimes\Bbbk_{M_\infty}^{\rmE, \sub},$$
see \cite[Prop.\:3.23]{Ito21b} for the details.
Moreover, for an analytic function $\varphi \colon U\to \RR$ defined on an open subset $U\subset M$,
we set %the exponential enhanced ind-sheaf by 
%\hypertarget{E}{\[\EE_{U|M_\infty}^\varphi \colon= 
%\CC_{M_\infty}^\rmE\Potimes
%\Q_{M_\infty}\big(\CC_{\{t+\varphi\geq0\}}\big), \]}
\[\EE_{U|M_\infty}^{\varphi ,\sub}:= 
\Bbbk_{M_\infty}^{\rmE,\sub}\Potimes
\Q_{M_\infty}^\sub\big(\Bbbk_{\{t+\varphi = 0\}}\big)
, \]
where $\{t+\varphi = 0\}$ stands for 
$\{(x, t)\in \che{M}\times\var{\RR}\ |\ t\in\RR, x\in U, 
t+\varphi(x) = 0\}$.

At the end of this subsection,
let us recall the definition of $\RR$-constructible enhanced subanalytic sheaves.
\begin{definition}\label{def3.21}
We say that an enhanced subanalytic sheaf $K$ is $\RR$-constructible 
if for any open subset $U$ of $M$ which is subanalytic and relatively compact in $\che{M}$
there exists $\SF\in\BDC_{\RR\mbox{\scriptsize -}c}(\Bbbk_{U_\infty\times\RR_\infty})$
such that $$\bfE i_{U_\infty}^{-1}K\simeq
\Bbbk_{U_\infty}^{\rmE, \sub}\Potimes \Q_{U_\infty}^\sub\rho_{U_\infty\times\RR_\infty\ast}\SF.$$
\end{definition}

We denote by $\BEC_{\RR\mbox{\scriptsize -}c}(\Bbbk_{M_\infty}^\sub)$
the full triangulated subcategory of $\BEC(\Bbbk_{M_\infty}^\sub)$
consisting of $\RR$-constructible enhanced subanalytic sheaves.

\subsection{Relation between Ind-sheaves and Subanalytic Sheaves 2}\label{2.11}
We shall briefly recall the relation between enhanced ind-sheaves and enhanced subanalytic sheaves.
References are made to \cite[\S\:3.4]{Ito21b}.

Let $\Bbbk$ be a field and $M_\infty$ a real analytic bordered space.
We set
\begin{align*}
I_{M_\infty}^\rmE&\colon\BEC(\Bbbk_{M_\infty}^\sub)\to\BEC(\I\Bbbk_{M_\infty}),
\hspace{7pt}
\Q_{M_\infty}^\sub\SF\mapsto \Q_{M_\infty}I_{M_\infty\times\RR_\infty}\SF,\\
J_{M_\infty}^\rmE&\colon\BEC(\I\Bbbk_{M_\infty})\to\BEC(\Bbbk_{M_\infty}^\sub),
\hspace{7pt}
\Q_{M_\infty}F\mapsto \Q_{M_\infty}^\sub\bfR J_{M_\infty\times\RR_\infty}F.
\end{align*} 
and
$$
\BEC_{\I\RR\mbox{\scriptsize -}c}(\I\Bbbk_{M_\infty}) :=
\BDC_{\I\RR\mbox{\scriptsize -}c}(\I\Bbbk_{M_\infty \times\RR_\infty})/\pi^{-1}\BDC_{\I\RR\mbox{\scriptsize -}c}(\I\Bbbk_{M_\infty}). 
$$
Note that $\BEC_{\I\RR\mbox{\scriptsize -}c}(\I\Bbbk_{M_\infty})$ is
a full triangulated subcategory of $\BEC_{\RR\mbox{\scriptsize -}c}(\I\Bbbk_{M_\infty})$.
Then we have:

\begin{theorem}[{\cite[Thms.\:3.17, 3.22]{Ito21b}}]\label{thm2.6}
\begin{itemize}
\item[\rm (1)]
A pair $(I_{M_\infty}^\rmE, J_{M_\infty}^\rmE)$ is an adjoint pair
and there exists a canonical isomorphism $\id\simto J_{M_\infty}^\rmE\circ I_{M_\infty}^\rmE$,

\item[\rm (2)]
There exists an equivalence of triangulated categories:
\[\xymatrix@M=7pt@C=45pt{
\BEC(\Bbbk_{M_\infty}^\sub)\ar@<0.8ex>@{->}[r]^-{I_{M_\infty}^\rmE}_-\sim
&
\BEC_{\I{\RR\mbox{\scriptsize -}c}}(\I\Bbbk_{M_\infty})
\ar@<0.8ex>@{->}[l]^-{J_{M_\infty}^\rmE}.
}\]

\item[\rm (3)]
there exists an equivalence of triangulated categories:
\[\xymatrix@M=7pt@C=45pt{
\BEC_{\RR\mbox{\scriptsize -}c}(\Bbbk_{M_\infty}^\sub)\ar@<0.8ex>@{->}[r]^-{I_{M_\infty}^\rmE}_-\sim
&
\BEC_{\RR\mbox{\scriptsize -}c}(\I\Bbbk_{M_\infty})
\ar@<0.8ex>@{->}[l]^-{J_{M_\infty}^\rmE}.
}\]
\end{itemize}
\end{theorem}

We will denote by $\lambda_{M_\infty}^\rmE$%\colon \BEC_{\I{\RR-c}}(\I\Bbbk_{M_\infty})\simto \BEC(\Bbbk_{M_\infty}^\sub)$
 the inverse functor of
$I_{M_\infty}^\rmE\colon \BEC(\Bbbk_{M_\infty}^\sub)\simto \BEC_{\I{\RR-c}}(\I\Bbbk_{M_\infty})$.
Let us summarize the relations between enhanced subanalytic sheaves and enhanced ind-sheaves
 in the following commutative diagram:
\[\xymatrix@M=10pt@R=30pt@C=85pt{
\BEC(\Bbbk_{M_\infty}^\sub)\ar@<0.7ex>@{^{(}->}[r]^-{I_{M_\infty}^\rmE}
\ar@<-0.7ex>@{->}[rd]^-{I_{M_\infty}^\rmE}_-\sim
 & \BEC(\I\Bbbk_{M_\infty})\ar@<0.7ex>@{->}[l]^-{J_{M_\infty}^\rmE}\\
\BEC_{\RR-c}(\Bbbk_{M_\infty}^\sub)\ar@{}[u]|-{\bigcup}
\ar@<0.7ex>@{->}[rd]^-{I_{M_\infty}^\rmE}_-\sim
& \BEC_{\I\RR-c}(\I\Bbbk_{M_\infty}).\ar@{}[u]|-{\bigcup}
\ar@<2.5ex>@{->}[lu]^-{\lambda_{M_\infty}^\rmE}\\
{} & \BEC_{\RR-c}(\I\Bbbk_{M_\infty})\ar@{}[u]|-{\bigcup}
\ar@<1.0ex>@{->}[lu]^-{\lambda_{M_\infty}^\rmE}.
}\]

Furthermore, we have commutative diagrams:
\[\xymatrix@M=7pt@C=45pt@R=35pt{
\BDC(\Bbbk_{M_\infty}^\sub)\ar@{->}[r]^-{e_{M_\infty}^\sub}\ar@{->}[d]_-{I_{M_\infty}}^-\wr
&
\BEC(\Bbbk_{M_\infty}^\sub)\ar@{->}[d]^-{I_{M_\infty}^\rmE}_-\wr\\
\BDC_{\I\RR-c}(\I\Bbbk_{M_\infty})\ar@{->}[r]_-{e_{M_\infty}}
&
\BEC_{\I\RR-c}(\I\Bbbk_{M_\infty}),}
\hspace{27pt}
\xymatrix@M=7pt@C=45pt@R=35pt{
\BDC_{\I\RR-c}(\I\Bbbk_{M_\infty})\ar@{->}[r]^-{e_{M_\infty}}\ar@{->}[d]_-{\lambda_{M_\infty}}^-\wr
&
\BEC_{\I\RR\mbox{\scriptsize -}c}(\I\Bbbk_{M_\infty})\ar@{->}[d]^-{\lambda_{M_\infty}^\rmE}_-\wr\\
\BDC(\Bbbk_{M_\infty}^\sub)\ar@{->}[r]_-{e_{M_\infty}^\sub}
&
\BEC(\Bbbk_{M_\infty}^\sub).}
\]

\subsection{$\CC$-Constructible Enhanced Ind-Sheaves}
We shall briefly recall the notion of $\CC$-constructability for enhanced ind-sheaves.
References are made to \cite[\S\:3]{Ito20} for the analytic case
and \cite[\S\S\:3.1. 3.2]{Ito21} for the algebraic case.

\subsubsection{Analytic Case}\label{subsec_ana}
Let $X$ be a complex manifold and 
$D \subset X$ a normal crossing divisor on it. 
Let us take local coordinates 
$(u_1, \ldots, u_l, v_1, \ldots, v_{d_X-l})$ 
of $X$ such that $D= \{ u_1 u_2 \cdots u_l=0 \}$
and set $Y= \{ u_1=u_2= \cdots =u_l=0 \}.$
We define a partial order $\leq$ on the 
set $\ZZ^l$ by 
$$a \leq a^{\prime} \ :\Longleftrightarrow 
\ a_i \leq a_i^{\prime} \ (1 \leq {}^\forall i \leq l),$$
for $a = (a_1, \ldots, a_l),\ a' = (a'_1, \ldots, a'_l) \in \ZZ^l$.
Then for a meromorphic function $\varphi\in\SO_{X}(\ast D)$
on $X$ along $D$ which have the Laurent expansion
\[ \varphi = \sum_{a \in \ZZ^l} c_a( \varphi )(v) \cdot 
u^a \ \in \SO_{X}(\ast D) \]
with respect to $u_1, \ldots, u_{l}$,
where $c_a( \varphi )$ are holomorphic functions on $Y$, 
we define its order 
$\ord( \varphi ) \in \ZZ^l$ by the minimum 
\[ \min \Big( \{ a \in \ZZ^l \ | \
c_a( \varphi ) \not= 0 \} \cup \{ 0 \} \Big) \]
if it exists. 
For any $f\in\SO_{X}(\ast D)/ \SO_{X}$, we take any lift $\tl{f}$ to $\SO_{X}(\ast D)$,
and we set $\ord(f) := \ord(\tl{f})$, if the right-hand side exists.
Note that it is independent of the choice of a lift $\tl{f}$.
If $\ord(f)\neq0$, $c_{\ord(f)}(\tl{f})$ is independent of the choice of a lift $\tl{f}$,
which is denoted by $c_{\ord(f)}(f)$.
\begin{definition}[{\cite[Def.\:2.1.2]{Mochi11}}]\label{def2.10}
In the situation as above,
a finite subset $\calI \subset \SO_{X}(\ast D)/ \SO_{X}$
is called a good set of irregular values on $(X,D)$,
if the following conditions are satisfied:
\begin{itemize}
\setlength{\itemsep}{-1pt}
\item[-]
For each element $f\in\calI$, $\ord(f)$ exists.
If $f\neq0$ in $\SO_{X}(\ast D)/ \SO_{X}$, $c_{\ord(f)}(f)$ is invertible on $Y$.
\item[-]
For two distinct $f, g\in\calI$, $\ord(f-g)$ exists and
$c_{\ord(f-g)}(f-g)$ is invertible on $Y$.
\item[-]
The set $\{\ord(f-g)\ |\ f, g\in\calI\}$ is totally ordered
with respect to the above partial order $\leq$ on $\ZZ^l$.
\end{itemize}
\end{definition}

\begin{definition}[{\cite[Def.\:3.5]{Ito20}}]\label{def-normal}
In the situation as above,
we say that an enhanced ind-sheaf
$K\in\ZEC(\I\CC_X)$ has a normal form along $D$
if the following three conditions are satisfied: 
\begin{itemize}
\setlength{\itemsep}{-1pt}
\item[(i)]
$\pi^{-1}(\iota_X\CC_{X\setminus D})\otimes K\simto K$,

\item[(ii)]
for any $x\in X\setminus D$ there exist an open neighborhood $U_x\subset X\setminus D$
of $x$ and a non-negative integer $k$ such that
$K|_{U_x}\simeq (\CC_{U_x}^{\rmE})^{\oplus k},$

\item[(iii)]
for any $x\in D$ there exist an open neighborhood $U_x\subset X$ of $x$,
a good set of irregular values $\{\varphi_i\}_i$ on $(U_x, D\cap U_x)$
and a finite sectorial open covering $\{U_{x, j}\}_j$ of $U_x\bs D$
such that
\[\pi^{-1}(\iota_{U_x}\CC_{U_{x, j}})\otimes K|_{U_x}\simeq
\bigoplus_i \EE_{U_{x, j} | U_x}^{\Re\varphi_i}
\hspace{10pt} \mbox{for any } j.\]
%see the end of \S\ref{subsec2.3} for the definition of 
%\hyperlink{E}{$\EE_{U_{x, j} | U_x}^{\Re\varphi_i}$}
%$\EE_{U_{x, j} | U_x}^{\Re\varphi_i}$.
\end{itemize}
\end{definition}
\newpage

In \cite[Def.\:3.5]{Ito20}, 
we assumed that $K$ is $\RR$-constructible.
However, it is not necessary.
Namely, any enhanced ind-sheaf which have a normal form along $D$ which is defined by above
is an $\RR$-constructible enhanced ind-sheaf.
See \cite[Prop.\:2.11]{Ito23} for the details.

A ramification of $X$ along a normal crossing divisor $D$ on a neighborhood $U$ 
of $x \in D$ is a finite map $r \colon U^{\rami}\to U$ of complex manifolds of the form
$z' \mapsto z=(z_1,z_2, \ldots, z_n)= (z'^{m_1}_1,\ldots, z'^{m_k}_k, z'_{k+1},\ldots,z'_n)$ 
for some $(m_1, \ldots, m_k)\in (\ZZ_{>0})^k$, where 
$(z'_1,\ldots, z'_n)$ is a local coordinate system of $U^{\rami}$ and 
$(z_1, \ldots, z_n)$ is the one of 
$U$ such that $D \cap U=\{z_1\cdots z_k=0\}$. 

\begin{definition}[{\cite[Def.\:3.11]{Ito20}}]\label{def-quasi}
We say that an enhanced ind-sheaf
$K\in\ZEC(\I\CC_{X})$ has a quasi-normal form along $D$
if it satisfies (i) and (ii) in Definition \ref{def-normal}, and if
for any $x\in D$ there exist an open neighborhood $U_x\subset X$ of $x$
and a ramification $r_x \colon U_x^{\rami}\to U_x$ of $U_x$ along $D_x := U_x\cap D$
such that $\bfE r_x^{-1}(K|_{U_x})$ has a normal form along $D_x^{\rami}:= r_x^{-1}(D_x)$.
\end{definition}
Note that any enhanced ind-sheaf which have a quasi-normal form along $D$
is an $\RR$-constructible enhanced ind-sheaf on $X$.
See \cite[Prop.\:3.12]{Ito20} for the details.

A modification of $X$ with respect to an analytic hypersurface $H$
is a projective map $m \colon X^{\modi}\to X$ from a complex manifold $X^{\modi}$ to $X$ such that
$D^{\modi} := m^{-1}(H)$ is a normal crossing divisor of $X^{\modi}$
and $m$ induces an isomorphism $X^{\modi}\setminus D^{\modi}\simto X\setminus H$.

\begin{definition}[{\cite[Def.\:3.14]{Ito20}}]\label{def-modi}
We say that an enhanced ind-sheaf $K\in\ZEC(\I\CC_{X})$
has a modified quasi-normal form along $H$
if it satisfies (i) and (ii) in Definition \ref{def-normal}, and if
for any $x\in H$ there exist an open neighborhood $U_x\subset X$ of $x$
and a modification $m_x \colon U_x^{\modi}\to U_x$ of $U_x$ along $H_x := U_x\cap H$
such that $\bfE m_x^{-1}(K|_{U_x})$ has a quasi-normal form along $D_x^{\modi} := m_x^{-1}(H_x)$.
\end{definition}
Note that any enhanced ind-sheaf which have a modified quasi-normal form along $H$
is an $\RR$-constructible enhanced ind-sheaf on $X$. 
See \cite[Prop.\:3.15]{Ito20} for the details. 

A complex analytic stratification of $X$ is
a locally finite partition $\{X_\alpha\}_{\alpha\in A}$ of $X$
by locally closed analytic subsets $X_\alpha$
such that for any $\alpha\in A$, $X_\alpha$ is smooth,
$\var{X}_\alpha$ and $\partial{X_\alpha} := \var{X}_\alpha\setminus X_\alpha$
are complex analytic subsets 
and $\var{X}_\alpha = \bigsqcup_{\beta\in B} X_{\beta}$ for a subset $B\subset A$.
\begin{definition}[{\cite[Def.\:3.19]{Ito20}}]\label{def-const}
We say that an enhanced ind-sheaf $K\in\ZEC(\I\CC_{X})$ is $\CC$-constructible if
there exists a complex analytic stratification $\{X_\alpha\}_\alpha$ of $X$
such that
$$\pi^{-1}(\iota_{\var{X}^{\blow}_\alpha}\CC_{\var{X}^{\blow}_\alpha\setminus D_\alpha})
\otimes \bfE b_\alpha^{-1}K$$
has a modified quasi-normal form along $D_\alpha$ for any $\alpha$.
Here $b_\alpha \colon \var{X}^{\blow}_\alpha \to X$ is a complex blow-up of $\var{X_\alpha}$
along $\partial X_\alpha = \var{X_\alpha}\setminus X_\alpha$ and
$D_\alpha := b_\alpha^{-1}(\partial X_\alpha)$.
Namely $\var{X}^{\blow} _\alpha$ is a complex manifold,
$D_\alpha$ is a normal crossing divisor of $\var{X}^{\blow} _\alpha$
and $b_\alpha$ is a projective map
which induces an isomorphism $\var{X}^{\blow} _\alpha\setminus D_\alpha\simto X_\alpha$
and satisfies $b_\alpha\big(\var{X}^{\blow} _\alpha\big)=\var{X_\alpha}$.
\end{definition}

%\begin{remark}\label{rem2.25}
%Definition \ref{def-const} does not depend on the choice of a complex blow-up $b_\alpha$
%by \cite[Sublem.\:3.22]{Ito20}.
%\end{remark}

We denote by $\ZEC_{\CC\mbox{\scriptsize -}c}(\I\CC_{X})$ the full subcategory of $\ZEC(\I\CC_{X})$
whose objects are $\CC$-constructible
and denote by $\BEC_{\CC\mbox{\scriptsize -}c}(\I\CC_{X})$
the full triangulated subcategory of $\BEC(\I\CC_{X})$
consisting of objects whose cohomologies are contained in $\ZEC_{\CC\mbox{\scriptsize -}c}(\I\CC_{X})$.
%\[\BEC_{\CC\mbox{\scriptsize -}c}(\I\CC_{X}) := \{K\in\BEC(\I\CC_{X})\
%|\ \SH^i(K)\in\ZEC_{\CC\mbox{\scriptsize -}c}(\I\CC_{X}) \mbox{ for any }i\in\ZZ \}\subset \BEC(\I\CC_{X}).\]
Note that the category $\BEC_{\CC\mbox{\scriptsize -}c}(\I\CC_{X})$ is
the full triangulated subcategory of $\BEC_{\RR\mbox{\scriptsize -}c}(\I\CC_{X})$.
See \cite[Prop.\:3.21]{Ito20} for the details.
Furthermore we have:
\begin{theorem}[{\cite[Prop.\:3.25, Thm.\:3.26]{Ito20}}]\label{thm-Ito20}
There exists an equivalence of triangulated categories:
$$\Sol_{X}^{\rmE}\colon \BDChol(\D_{X})^{\op}\simto \BEC_{\CC\mbox{\scriptsize -}c}(\I\CC_{X}).$$
\end{theorem}
\newpage

\subsubsection{Algebraic Case}
Let $X$ be a smooth algebraic variety over $\CC$ and
denote by $X^\an$ the underlying complex manifold of $X$.
Recall that an algebraic stratification of $X$ is
a Zariski locally finite partition $\{X_\alpha\}_{\alpha\in A}$ of $X$
by locally closed subvarieties $X_\alpha$
such that for any $\alpha\in A$, $X_\alpha$ is smooth and
$\var{X}_\alpha = \bigsqcup_{\beta\in B} X_{\beta}$ for a subset $B\subset A$.
Moreover an algebraic stratification $\{X_\alpha\}_{\alpha\in A}$ of $X$ induces
a complex analytic stratification $\{X^\an_\alpha\}_{\alpha\in A}$ of $X^\an$.

\begin{definition}[{\cite[Thm.\:3.1]{Ito21}}]
We say that an enhanced ind-sheaf $K\in\ZEC(\I\CC_{X^\an})$ satisfies the condition 
%\hypertarget{AC}{$\AC$}
$\AC$
if there exists an algebraic stratification $\{X_\alpha\}_\alpha$ of $X$ such that
$$\pi^{-1}(\iota_{(\var{X}^{\blow}_\alpha)^\an}\CC_{(\var{X}^{\blow}_\alpha)^\an \setminus D^\an_\alpha})
\otimes \bfE (b^\an_\alpha)^{-1}K$$
has a modified quasi-normal form along $D^\an_\alpha$ for any $\alpha$,
Here $b_\alpha \colon \var{X}^\blow_\alpha \to X$ is a blow-up of $\var{X_\alpha}$
along $\partial X_\alpha := \var{X_\alpha}\setminus X_\alpha$,
$D_\alpha := b_\alpha^{-1}(\partial X_\alpha)$
and $D_\alpha^\an := \big(\var{X}_\alpha^\blow\big)^\an\setminus
\big(\var{X}_\alpha^\blow\setminus D_\alpha\big)^\an$.
Namely $\var{X}^\blow_\alpha$ is a smooth algebraic variety over $\CC$,
$D_\alpha$ is a normal crossing divisor of $\var{X}^\blow_\alpha$
and $b_\alpha$ is a projective map
which induces an isomorphism $\var{X}^\blow_\alpha\setminus D_\alpha\simto X_\alpha$
and satisfies $b_\alpha\big(\var{X}^\blow_\alpha\big)=\var{X_\alpha}$.
\end{definition}

We denote by $\ZEC_{\CC\mbox{\scriptsize -}c}(\I\CC_X)$ the full subcategory of $\ZEC(\I\CC_{X^\an})$
whose objects satisfy the condition 
%\hyperlink{AC}{$\AC$}
$\AC$.
Note that $\ZEC_{\CC\mbox{\scriptsize -}c}(\I\CC_X)$ is the full subcategory of
the abelian category $\ZEC_{\CC\mbox{\scriptsize -}c}(\I\CC_{X^\an})$
of $\CC$-constructible enhanced ind-sheaves on $X^\an$.
Moreover we %set
%\[\BEC_{\CC\mbox{\scriptsize -}c}(\I\CC_X) := \{K\in\BEC(\I\CC_{X^\an})\
%|\ \SH^i(K)\in\ZEC_{\CC\mbox{\scriptsize -}c}(\I\CC_X) \mbox{ for any }i\in\ZZ \}
%\subset \BEC_{\CC\mbox{\scriptsize -}c}(\I\CC_{X^\an}).\]
denote by $\BEC_{\CC\mbox{\scriptsize -}c}(\I\CC_{X})$
the full triangulated subcategory of $\BEC(\I\CC_{X})$
consisting of objects whose cohomologies are contained in $\ZEC_{\CC\mbox{\scriptsize -}c}(\I\CC_{X})$.

\begin{theorem}[{\cite[Thm.\:3.7]{Ito21}}]\label{thm-Ito21_complete}
Let $X$ be a smooth complete algebraic variety over $\CC$.
Then there exists an equivalence of triangulated categories
\[\Sol_X^{\rmE} \colon \BDChol(\D_X)^{\op}\simto \BEC_{\CC\mbox{\scriptsize -}c}(\I\CC_X),\
\M\mapsto \Sol_X^\rmE(\M) := \Sol_{X^\an}^\rmE(\M^\an).\]
\end{theorem}

Thanks to Hironaka's desingularization theorem \cite{Hiro64} (see also \cite[Thm\:4.3]{Naga62}),
we can take a smooth complete algebraic variety $\tl{X}$ such that $X\subset \tl{X}$
and $D := \tl{X}\setminus X$ is a normal crossing divisor of $\tl{X}$.
Let us consider a bordered space $X^\an_\infty = (X^\an, \tl{X}^\an)$
and the triangulated category $\BEC(\I\CC_{X^\an_\infty})$ of enhanced ind-sheaves 
on $X^\an_\infty$.
Remark that $\BEC(\I\CC_{X^\an_\infty})$ does not depend on the choice of $\tl{X}$,
see \cite[\S 2.3]{Ito21} for the details.

We shall denote by $j\colon X\hookrightarrow \tl{X}$ the open embedding,
and by $j^\an\colon X^\an\hookrightarrow \tl{X}^\an$ the correspondence morphism for analytic spaces of $j$. 
Then we obtain the morphism of bordered spaces
$j^\an\colon X^\an_\infty\to \tl{X}^\an$
given by the embedding $j^\an\colon X^\an\hookrightarrow \tl{X}^\an$.

\begin{definition}[{\cite[Def.\:3.10]{Ito21}}]\label{def-algconst}
We say that an enhanced ind-sheaf $K\in\BEC(\I\CC_{X^\an_\infty})$ is
algebraic $\CC$-constructible on $X_\infty^\an$
if $\bfE j^\an_{!!}K \in\BEC_{\CC\mbox{\scriptsize -}c}(\I\CC_{\tl{X}})$.
\end{definition}

We denote by $\BEC_{\CC\mbox{\scriptsize -}c}(\I\CC_{X_\infty})$
the full triangulated subcategory of $\BEC(\I\CC_{X^\an_\infty})$
consisting of algebraic $\CC$-constructible enhanced ind-sheaves on $X_\infty^\an$.
Note that the triangulated category $\BEC_{\CC\mbox{\scriptsize -}c}(\I\CC_{X_\infty})$ is the full triangulated subcategory
of $\BEC_{\RR\mbox{\scriptsize -}c}(\I\CC_{X^\an_\infty})$.
Furthermore we have:

\begin{theorem}[{\cite[Thm.\:3.11]{Ito21}}]\label{thm-Ito21}
There exists an equivalence of triangulated categories:
$$\Sol_{X_\infty}^{\rmE} \colon \BDChol(\D_X)^{\op}\simto \BEC_{\CC\mbox{\scriptsize -}c}(\I\CC_{X_\infty}),\
\M\mapsto \bfE j_{X_\infty^\an}^{-1}\Sol_{\tl{X}}^{\rmE, \sub}(\bfD j_\ast\M).$$
Here $\bfD j_!$ is the proper direct image functor for algebraic $\D$-modules by $j$.
\end{theorem}

\section{Main Results}\label{sec3}

\subsection{Relation between $\ZEC(\Bbbk_{M_\infty}^\sub)$ and $\ZEC(\I\Bbbk_{M_\infty})$}
In this subsection, we shall explain a relation between 
$\ZEC(\Bbbk_{M_\infty}^\sub)$ and $\ZEC(\I\Bbbk_{M_\infty})$.

Throughout this subsection,
let $\Bbbk$ be a field and $M_\infty$ a real analytic bordered spaces.
Let us set
\begin{align*}
\bfD_{\I{\RR\mbox{\scriptsize -}c}}^{\leq 0}(\I\Bbbk_{M_\infty}) &:= 
\bfD^{\leq 0}(\I\Bbbk_{M_\infty})\cap
\BDC_{\I{\RR\mbox{\scriptsize -}c}}(\I\Bbbk_{M_\infty}),\\
\bfD_{\I{\RR\mbox{\scriptsize -}c}}^{\geq 0}(\I\Bbbk_{M_\infty}) &:= 
\bfD^{\geq 0}(\I\Bbbk_{M_\infty})\cap
\BDC_{\I{\RR\mbox{\scriptsize -}c}}(\I\Bbbk_{M_\infty}).
\end{align*}

\begin{lemma}
A pair $(\bfD_{\I{\RR\mbox{\scriptsize -}c}}^{\leq 0}(\I\Bbbk_{M_\infty}),
\bfD_{\I{\RR\mbox{\scriptsize -}c}}^{\geq 0}(\I\Bbbk_{M_\infty}))$
is a t-structure on $\BDC_{\I{\RR\mbox{\scriptsize -}c}}(\I\Bbbk_{M_\infty})$.
\end{lemma}

\begin{proof}
It is enough to show that for any $F\in \BDC_{\I{\RR\mbox{\scriptsize -}c}}(\I\Bbbk_{M_\infty})$
there exists a distinguished triangle in $\BDC_{\I{\RR\mbox{\scriptsize -}c}}(\I\Bbbk_{M_\infty})$
$$F_1 \to F \to F_2 \xrightarrow{+1}$$
with $F_1\in \bfD^{\leq 0}_{\I{\RR\mbox{\scriptsize -}c}}(\I\Bbbk_{M_\infty}),
F_2\in \bfD^{\geq 1}_{\I{\RR\mbox{\scriptsize -}c}}(\I\Bbbk_{M_\infty})$.
Let $F\in \BDC_{\I{\RR\mbox{\scriptsize -}c}}(\I\Bbbk_{M_\infty})$.
Since $\bfR j_{M_\infty!!}F\in \BDC_{\I{\RR\mbox{\scriptsize -}c}}(\I\Bbbk_{\che{M}})$,
there exists 
a distinguished triangle in $\BDC_{\I{\RR\mbox{\scriptsize -}c}}(\I\Bbbk_{\che{M}})$.
$$F_1' \to \bfR j_{M_\infty!!}F \to F_2' \xrightarrow{+1}$$
with $F_1'\in \bfD^{\leq 0}_{\I{\RR\mbox{\scriptsize -}c}}(\Bbbk_{\che{M}}^\sub),
F_2'\in \bfD^{\geq 1}_{\I{\RR\mbox{\scriptsize -}c}}(\Bbbk_{\che{M}}^\sub)$
and hence we have a distinguished triangle in $\BDC(\I\Bbbk_{M_\infty})$
$$j_{M_\infty}^{-1}F_1' \to F \to j_{M_\infty}^{-1}F_2' \xrightarrow{+1}.$$

Let us set $F_1 := j_{M_\infty}^{-1}F_1'$ and $F_2 := j_{M_\infty}^{-1}F_2'$.
Then we have isomorphisms in $\BDC(\I\Bbbk_{\che{M}})$
$$\bfR j_{M_\infty!!}F_1 \simeq \iota_{\che{M}}\Bbbk_M\otimes F'_1,\
\bfR j_{M_\infty!!}F_2 \simeq \iota_{\che{M}}\Bbbk_M\otimes F'_2$$
by \cite[Lem.\:3.3.7 (ii)]{DK16}
and hence we have
$$\bfR j_{M_\infty!!}F_1\in \bfD^{\leq 0}_{\I{\RR\mbox{\scriptsize -}c}}(\I\Bbbk_{M_\infty}),\ 
\bfR j_{M_\infty!!}F_2\in \bfD^{\geq 1}_{\I{\RR\mbox{\scriptsize -}c}}(\I\Bbbk_{M_\infty})$$
by \cite[Prop.\:3.4.3 (i)]{DK16} and \cite[Lem.\:3.5 (1), (3)]{Ito21b}.
Therefore we have a distinguished triangle in $\BDC_{\I{\RR\mbox{\scriptsize -}c}}(\I\Bbbk_{M_\infty})$
$$F_1 \to F \to F_2 \xrightarrow{+1}$$
with $F_1\in \bfD^{\leq 0}_{\I{\RR\mbox{\scriptsize -}c}}(\Bbbk_{M_\infty}^\sub),
F_2\in \bfD^{\geq 1}_{\I{\RR\mbox{\scriptsize -}c}}(\Bbbk_{M_\infty}^\sub)$.
\end{proof}

We shall set
$\I_{\RR\mbox{\scriptsize -}c}(\I\Bbbk_{M_\infty}) :=
\bfD_{\I{\RR\mbox{\scriptsize -}c}}^{\leq 0}(\I\Bbbk_{M_\infty})\cap
\bfD_{\I{\RR\mbox{\scriptsize -}c}}^{\geq 0}(\I\Bbbk_{M_\infty}).$

\begin{lemma}\label{lem3.2}
The functor $I_{M_\infty}\colon \BDC(\Bbbk_{M_\infty}^\sub)\to\BDC(\I\Bbbk_{M_\infty})$
is t-exact with respect to the standard t-structures.
In particular, the functors $I_{M_\infty}, \lambda_{M_\infty}$ induce an equivalence of abelian categories:
\[\xymatrix@M=7pt@C=45pt{
\Mod(\Bbbk_{M_\infty}^\sub)\ar@<0.8ex>@{->}[r]^-{I_{M_\infty}}_-\sim
&
\I_{\RR\mbox{\scriptsize -}c}(\I\Bbbk_{M_\infty})
\ar@<0.8ex>@{->}[l]^-{\lambda_{M_\infty}}.
}\]
\end{lemma}

\begin{proof}
By Theorem \ref{thm2.2} (2) and \cite[Prop.\:10.1.6]{KS90},
it is enough to show that the functor $I_{M_\infty}$
is t-exact with respect to the standard t-structures.

Recall that for any $\SF\in \BDC(\Bbbk_{M_\infty}^\sub)$
we have isomorphisms in $\BDC(\I\Bbbk_{\che{M}})$
$$\bfR j_{M_\infty!!}I_{M_\infty}(\SF) \simeq
\bfR j_{M_\infty!!}(\q_{M_\infty}I_{\che{M}}\bfR j_{M_\infty!!}\SF)\\
\simeq
\iota_{\che{M}}\Bbbk_M\otimes I_{\che{M}}\bfR j_{M_\infty!!}\SF.$$
%and hence $\bfR j_{M_\infty!!}I_{M_\infty}(\SF) \in \BDC_{\I{\RR\mbox{\scriptsize -}c}}(\I\Bbbk_{\che{M}})$.
Moreover, since the functor $\bfR j_{M_\infty!!}$ is left adjoint functor of
t-exact functor $j_{M_\infty}^{-1}$ with respect to the standard t-structures,
it is right t-exact with respect to them 
by \cite[Cor.\:10.1.18]{KS90}
and hence it is t-exact with respect to them.

Let $\SF\in \bfD^{\leq 0}(\Bbbk_{M_\infty}^\sub)$.
Then we have $\bfR j_{M_\infty!!}\SF\in \bfD^{\leq 0}(\Bbbk_{\che{M}}^\sub)$
and hence we have $\bfR j_{M_\infty!!}I_{M_\infty}(\SF)\in \bfD^{\leq 0}(\I\Bbbk_{\che{M}})$.
This implies that $I_{M_\infty}(\SF)\in \bfD^{\leq 0}(\I\Bbbk_{M_\infty})$
and hence we have $I_{M_\infty}(\SF)\in \bfD_{\I{\RR\mbox{\scriptsize -}c}}^{\leq 0}(\I\Bbbk_{M_\infty})$.

Let $\SF\in \bfD^{\geq 1}(\Bbbk_{M_\infty}^\sub)$.
Then we have $\bfR j_{M_\infty!!}\SF\in \bfD^{\geq 1}(\Bbbk_{\che{M}}^\sub)$
and hence we have $\bfR j_{M_\infty!!}I_{M_\infty}(\SF)\in \bfD^{\geq 1}(\I\Bbbk_{\che{M}})$.
This implies that $I_{M_\infty}(\SF)\in \bfD^{\geq 1}(\I\Bbbk_{M_\infty})$
and hence we have $I_{M_\infty}(\SF)\in \bfD_{\I{\RR\mbox{\scriptsize -}c}}^{\geq 1}(\I\Bbbk_{M_\infty})$. 
\end{proof}

Let us set
\begin{align*}
\bfE_{\I{\RR\mbox{\scriptsize -}c}}^{\leq 0}(\I\Bbbk_{M_\infty}) &:= 
\bfE^{\leq 0}(\I\Bbbk_{M_\infty})\cap
\BEC_{\I{\RR\mbox{\scriptsize -}c}}(\I\Bbbk_{M_\infty}),\\
\bfE_{\I{\RR\mbox{\scriptsize -}c}}^{\geq 0}(\I\Bbbk_{M_\infty}) &:= 
\bfE^{\geq 0}(\I\Bbbk_{M_\infty})\cap
\BEC_{\I{\RR\mbox{\scriptsize -}c}}(\I\Bbbk_{M_\infty}).
\end{align*}

\begin{lemma}
A pair $(\bfE_{\I{\RR\mbox{\scriptsize -}c}}^{\leq 0}(\I\Bbbk_{M_\infty}),
\bfE_{\I{\RR\mbox{\scriptsize -}c}}^{\geq 0}(\I\Bbbk_{M_\infty}))$
is a t-structure on $\BEC_{\I{\RR\mbox{\scriptsize -}c}}(\I\Bbbk_{M_\infty})$.
\end{lemma}

\begin{proof}
By the definition of the t-structure $(\bfE^{\leq 0}(\I\Bbbk_{M_\infty}), \bfE^{\geq 0}(\I\Bbbk_{M_\infty}))$
on $\BEC(\I\Bbbk_{M_\infty})$,
it is enough to show that 
$\bfL_{M_\infty}^\rmE K\in \BDC_{\I{\RR\mbox{\scriptsize -}c}}(\I\Bbbk_{M_\infty\times \mathbb{R}_\infty})$
for any $K\in \BEC_{\I{\RR\mbox{\scriptsize -}c}}(\I\Bbbk_{M_\infty})$.
This follows from the definition of $\BEC_{\I{\RR\mbox{\scriptsize -}c}}(\I\Bbbk_{M_\infty})$.
See \S\ref{2.11} for the details.
\end{proof}

We shall set
$\ZEC_{\I{\RR\mbox{\scriptsize -}c}}(\I\Bbbk_{M_\infty}) :=
\bfE_{\I{\RR\mbox{\scriptsize -}c}}^{\leq 0}(\I\Bbbk_{M_\infty})\cap
\bfE_{\I{\RR\mbox{\scriptsize -}c}}^{\geq 0}(\I\Bbbk_{M_\infty}).$

\begin{proposition}\label{prop3.4}
The functor $I_{M_\infty}^\rmE\colon \BEC(\Bbbk_{M_\infty}^\sub)\to\BEC(\I\Bbbk_{M_\infty})$
is t-exact with respect to the standard t-structures.
In particular, the functors $I_{M_\infty}^\rmE, \lambda_{M_\infty}^\rmE$ induce an equivalence of abelian categories:
\[\xymatrix@M=7pt@C=45pt{
\ZEC(\Bbbk_{M_\infty}^\sub)\ar@<0.8ex>@{->}[r]^-{I_{M_\infty}^\rmE}_-\sim
&
\ZEC_{\I{\RR\mbox{\scriptsize -}c}}(\I\Bbbk_{M_\infty})
\ar@<0.8ex>@{->}[l]^-{\lambda_{M_\infty}^\rmE}.
}\]
\end{proposition}

\begin{proof}
By Theorem \ref{thm2.6} (2) and \cite[Prop.\:10.1.6]{KS90},
it is enough to show that the functor $I_{M_\infty}^\rmE$
is t-exact with respect to the standard t-structures.

Remark that for any $K\in \BEC(\Bbbk_{M_\infty}^\sub)$
we have isomorphisms in $\BDC(\I\Bbbk_{M_\infty\times \RR_\infty})$
\begin{align*}
\bfL_{M_\infty}^\rmE I_{M_\infty}^\rmE(K) &\simeq
\bfL_{M_\infty}^\rmE(\Q_{M_\infty}I_{M_\infty\times \RR_\infty}\bfL_{M_\infty}^{\rmE, \sub}K)\\
&\simeq
\iota_{M_\infty\times\RR_\infty}(\Bbbk_{\{t\leq0\}}\oplus\Bbbk_{\{t\geq0\}})\Potimes
I_{M_\infty\times \RR_\infty}\bfL_{M_\infty}^{\rmE, \sub}K\\
&\simeq
I_{M_\infty\times\RR_\infty}(\rho_{M_\infty\times\RR_\infty\ast}(\Bbbk_{\{t\leq0\}}\oplus\Bbbk_{\{t\geq0\}})
\Potimes \bfL_{M_\infty}^{\rmE, \sub}K)\\
&\simeq
I_{M_\infty\times\RR_\infty}\bfL_{M_\infty}^{\rmE, \sub}(K)
\end{align*}
by \cite[Props.\:3.7 (4)(i), 3.12 (1)]{Ito21b}.

Let $K\in \bfE^{\leq 0}(\Bbbk_{M_\infty}^\sub)$.
Then we have $\bfL_{M_\infty}^{\rmE, \sub}K\in \bfD^{\leq0}(\Bbbk_{M_\infty\times \RR_\infty}^\sub)$
and hence we have 
$I_{M_\infty\times\RR_\infty}\bfL_{M_\infty}^{\rmE, \sub}(K)\in
\bfD_{\I{\RR\mbox{\scriptsize -}c}}^{\leq 0}(\I\Bbbk_{M_\infty\times\RR_\infty})$
by Lemma \ref{lem3.2}.
Therefore, we have $\bfL_{M_\infty}^\rmE I_{M_\infty}^\rmE(K)\in
\bfD_{\I{\RR\mbox{\scriptsize -}c}}^{\leq 0}(\I\Bbbk_{M_\infty\times\RR_\infty})$.
This implies that $I_{M_\infty}^\rmE(K)\in
\bfE_{\I{\RR\mbox{\scriptsize -}c}}^{\leq 0}(\I\Bbbk_{M_\infty})$.

By the same argument, for any  $K\in \bfE^{\geq 0}(\Bbbk_{M_\infty}^\sub)$
we have $I_{M_\infty}^\rmE(K)\in
\bfE_{\I{\RR\mbox{\scriptsize -}c}}^{\geq 0}(\I\Bbbk_{M_\infty})$.
%Let $K\in \bfE^{\geq 0}(\Bbbk_{M_\infty}^\sub)$.
%Then we have $\bfL_{M_\infty}^{\rmE, \sub}K\in \bfD^{\geq0}(\Bbbk_{M_\infty\times \RR_\infty}^\sub)$
%and hence we have 
%$I_{M_\infty\times\RR_\infty}\bfL_{M_\infty}^{\rmE, \sub}(K)\in
%\bfD_{\I{\RR\mbox{\scriptsize -}c}}^{\geq 0}(\I\Bbbk_{M_\infty})$
%by Lemma \ref{lem3.2}.
%Therefore, we have $\bfL_{M_\infty}^\rmE I_{M_\infty}^\rmE(K)\in
%\bfD_{\I{\RR\mbox{\scriptsize -}c}}^{\geq 0}(\I\Bbbk_{M_\infty})$.
%This implies that $I_{M_\infty}^\rmE(K)\in
%\bfE_{\I{\RR\mbox{\scriptsize -}c}}^{\geq 0}(\I\Bbbk_{M_\infty})$.
\end{proof}

Let us set
\begin{align*}
\bfE_{\RR\mbox{\scriptsize -}c}^{\leq 0}(\Bbbk_{M_\infty}^\sub) &:= 
\bfE^{\leq 0}(\Bbbk_{M_\infty}^\sub)\cap
\BEC_{\RR\mbox{\scriptsize -}c}(\I\Bbbk_{M_\infty}),\\
\bfE_{\RR\mbox{\scriptsize -}c}^{\geq 0}(\Bbbk_{M_\infty}^\sub) &:= 
\bfE^{\geq 0}(\Bbbk_{M_\infty}^\sub)\cap
\BEC_{\RR\mbox{\scriptsize -}c}(\I\Bbbk_{M_\infty}).
\end{align*}

\begin{lemma}
A pair $(\bfE_{\RR\mbox{\scriptsize -}c}^{\leq 0}(\Bbbk_{M_\infty}^\sub),
\bfE_{\RR\mbox{\scriptsize -}c}^{\geq 0}(\Bbbk_{M_\infty}^\sub))$
is a t-structure on $\BEC_{\RR\mbox{\scriptsize -}c}(\Bbbk_{M_\infty}^\sub)$.
\end{lemma}

\begin{proof}
It is enough to show that for any $K\in \BEC_{\RR\mbox{\scriptsize -}c}(\Bbbk_{M_\infty}^\sub)$
there exists a distinguished triangle in $\BEC_{\RR\mbox{\scriptsize -}c}(\Bbbk_{M_\infty}^\sub)$
$$K_1 \to K \to K_2 \xrightarrow{+1},$$
with $K_1\in \bfE^{\leq 0}_{\RR\mbox{\scriptsize -}c}(\Bbbk_{M_\infty}^\sub),
K_2\in \bfE^{\geq 1}_{\RR\mbox{\scriptsize -}c}(\Bbbk_{M_\infty}^\sub)$.
Let $K\in \BEC_{\RR\mbox{\scriptsize -}c}(\Bbbk_{M_\infty}^\sub)$.
Then there exists a distinguished triangle in $\BEC(\Bbbk_{M_\infty}^\sub)$
$$\tau^{\leq 0}K \to K \to \tau^{\geq 1}K \xrightarrow{+1}.$$
Here, $\tau^\bullet$ is the truncation functor with respect to
 $(\bfE^{\leq 0}(\Bbbk_{M_\infty}^\sub), \bfE^{\geq 0}(\Bbbk_{M_\infty}^\sub))$.
 Let us prove that $\tau^{\leq 0}K\in \bfE^{\leq 0}_{\RR\mbox{\scriptsize -}c}(\Bbbk_{M_\infty}^\sub)$
and $\tau^{\geq 1}K\in \bfE^{\geq 1}_{\RR\mbox{\scriptsize -}c}(\Bbbk_{M_\infty}^\sub)$.
 It is enough to show that $\tau^{\geq 1}K\in \bfE^{\geq 1}_{\RR\mbox{\scriptsize -}c}(\Bbbk_{M_\infty}^\sub)$.
 Let $U$ be an open subset of $M$ which is subanalytic and relatively compact on $\che{M}$. 
 Since $K\in \BEC_{\RR\mbox{\scriptsize -}c}(\Bbbk_{M_\infty}^\sub)$,
 there exists $\SF\in \BDC_{\RR\mbox{\scriptsize -}c}(\Bbbk_{M_\infty\times\RR_\infty})$
 such that $$\bfE i_{U_\infty}^{-1}K\simeq
 \Bbbk_{U_\infty}^{\rmE, \sub}\Potimes \Q_{U_\infty}^\sub\rho_{U_\infty\times\RR_\infty\ast}\SF.$$
 Since functors $\bfE i_{U_\infty}^{-1}$ and 
$\Bbbk_{U_\infty}^{\rmE, \sub}\Potimes \Q_{U_\infty}^\sub\rho_{U_\infty\times\RR_\infty\ast}(\cdot)$
 are left t-exact,
 we have isomorphisms in $\BEC(\Bbbk_{UM_\infty}^\sub)$
 \begin{align*}
 \bfE i_{U_\infty}^{-1}(\tau^{\geq 1}K) &\simeq \tau^{\geq 1}(\bfE i_{U_\infty}^{-1}K)\\
 &\simeq
 \tau^{\geq 1}(\Bbbk_{U_\infty}^{\rmE, \sub}\Potimes \Q_{U_\infty}^\sub\rho_{U_\infty\times\RR_\infty\ast}\SF)\\
 &\simeq 
\Bbbk_{U_\infty}^{\rmE, \sub}\Potimes \Q_{U_\infty}^\sub\rho_{U_\infty\times\RR_\infty\ast}(\tau^{\geq 1}\SF).
 \end{align*}
This implies that 
$\tau^{\geq 1}K\in \bfE^{\geq 1}_{\RR\mbox{\scriptsize -}c}(\Bbbk_{M_\infty}^\sub)$. 
\end{proof}

We shall set
$\ZEC_{\RR\mbox{\scriptsize -}c}(\Bbbk_{M_\infty}^\sub) :=
\bfE_{\RR\mbox{\scriptsize -}c}^{\leq 0}(\Bbbk_{M_\infty}^\sub)\cap
\bfE_{\RR\mbox{\scriptsize -}c}^{\geq 0}(\Bbbk_{M_\infty}^\sub).$

\begin{proposition}\label{prop3.6}
The functors $I_{M_\infty}^\rmE, \lambda_{M_\infty}^\rmE$ induce an equivalence of abelian categories:
\[\xymatrix@M=7pt@C=45pt{
\ZEC_{\RR\mbox{\scriptsize -}c}(\Bbbk_{M_\infty}^\sub)\ar@<0.8ex>@{->}[r]^-{I_{M_\infty}^\rmE}_-\sim
&
\ZEC_{\RR\mbox{\scriptsize -}c}(\I\Bbbk_{M_\infty})
\ar@<0.8ex>@{->}[l]^-{\lambda_{M_\infty}^\rmE}.
}\]
\end{proposition}

\begin{proof}
This follows from Proposition \ref{prop3.4}, Theorem \ref{thm2.6} (3).
\end{proof}

\subsection{$\CC$-Constructible Enhanced Subanalytic Sheaves}
Throughout of this subsection, let $X$ be a complex manifold.

\begin{definition}\label{def-normalsub}
We say that an enhanced subanalytic sheaf
$K\in\ZEC(\CC_X^\sub)$ has a normal form along a normal crossing divisor $D$
if the following three conditions are satisfied: 
\begin{itemize}
\setlength{\itemsep}{-1pt}
\item[(i)]
$\pi^{-1}(\rho_{X\ast}\CC_{X\setminus D})\otimes K\simto K$,

\item[(ii)]
for any $x\in X\setminus D$ there exist an open neighborhood $U_x\subset X\setminus D$
of $x$ and a non-negative integer $k$ such that
$K|_{U_x}\simeq (\CC_{U_x}^{\rmE,\sub})^{\oplus k},$

\item[(iii)]
for any $x\in D$ there exist an open neighborhood $U_x\subset X$ of $x$,
a good set of irregular values $\{\varphi_i\}_i$ on $(U_x, D\cap U_x)$
and a finite sectorial open covering $\{U_{x, j}\}_j$ of $U_x\bs D$
such that for any $j$
\[\pi^{-1}(\rho_{U_x\ast}\CC_{U_{x, j}})\otimes K|_{U_x}\simeq
\bigoplus_i \EE_{U_{x, j} | U_x}^{\Re\varphi_i, \sub}.\]
See \S\ref{subsec_ana} for the definition of good set of irregular values
and \S\ref{subsec2.10} for the definition of 
$\EE_{U_{x, j} | U_x}^{\Re\varphi_i,\sub}$.
\end{itemize}
\end{definition}

We shall denote by $\ZEC_{\rm norm}(\I\CC_{X(D)})$
the category of enhanced ind-sheaves which have normal form along $D$
and denote by $\ZEC_{\rm norm}(\CC_{X(D)}^\sub)$
the category of enhanced subanalytic sheaves which have normal form along $D$.
Remark that $\ZEC_{\rm norm}(\I\CC_{X(D)})$ is subcategory of
$\ZEC_{\RR\mbox{\scriptsize -}c}(\I\CC_X)$.
See \cite[Prop.\:2.11]{Ito23} for the details.

\begin{proposition}\label{prop3.8}
The functors $I_X^\rmE, \lambda_X^\rmE$
induce an equivalence of categories:
\[\xymatrix@M=7pt@C=45pt{
\ZEC_{\rm norm}(\CC_{X(D)}^\sub)\ar@<0.8ex>@{->}[r]^-{I_{X}^\rmE}_-\sim
&
\ZEC_{\rm norm}(\I\CC_{X(D)})
\ar@<0.8ex>@{->}[l]^-{\lambda_{X}^\rmE}.
}\]
\end{proposition}

\begin{proof}
By Proposition \ref{prop3.4},
it is enough to show that 
$I_X^\rmE(\,\ZEC_{\rm norm}(\CC_{X(D)}^\sub)\,)\subset \ZEC_{\rm norm}(\I\CC_{X(D)})$
and 
$\lambda_X^\rmE(\,\ZEC_{\rm norm}(\I\CC_{X(D)})\,)\subset \ZEC_{\rm norm}(\CC_{X(D)}^\sub)$.

Let $K\in \ZEC_{\rm norm}(\CC_{X(D)}^\sub)$.
Then we have an isomorphism
$\pi^{-1}(\rho_{X\ast}\CC_{X\setminus D})\otimes K\simto K$ in $\BEC(\CC_X^\sub)$
and hence there exist isomorphisms in $\BEC(\I\CC_X)$
$$\pi^{-1}(\iota_{X}\CC_{X\setminus D})\otimes I_X^\rmE K
\simeq 
I_X^\rmE(\pi^{-1}(\rho_{X\ast}\CC_{X\setminus D})\otimes K)
\simeq
I_X^\rmE K,$$
where in the second isomorphism we used \cite[Prop.\:3.7 (2)(iii),(vi), (4)(i)]{Ito21b}.
This implies that $I_X^\rmE K$ satisfies the condition (i) in Definition \ref{def-normal}.
Let $x\in X\setminus D$.
Then there exists an open neighborhood $U_x\subset X\setminus D$
of $x$ and a non-negative integer $k$ such that
$K|_{U_x}\simeq (\CC_{U_x}^{\rmE,\sub})^{\oplus k}$
and hence we have isomorphism in $\BEC(\I\CC_{U_x})$
$$(I_X^\rmE K)|_{U_x}
\simeq
I_{U_x}^\rmE(K|_{U_x})
\simeq
I_{U_x}^\rmE((\CC_{U_x}^{\rmE,\sub})^{\oplus k})
\simeq
(\CC_{U_x}^{\rmE})^{\oplus k},$$
where in the first isomorphism we used \cite[Prop.\:3.19 (2)(i)]{Ito21b}
and the third isomorphism we used the fact that
$I_{U_x}^\rmE(\CC_{U_x}^{\rmE,\sub}) \simeq \CC_{U_x}^{\rmE}$.
This implies that $I_X^\rmE K$ satisfies the condition (ii) in Definition \ref{def-normal}.
Let $x\in D$. Then there exist an open neighborhood $U_x\subset X$ of $x$,
a good set of irregular values $\{\varphi_i\}_i$ on $(U_x, D\cap U_x)$
and a finite sectorial open covering $\{U_{x, j}\}_j$ of $U_x\bs D$
such that for any $j$
\[\pi^{-1}(\rho_{U_x\ast}\CC_{U_{x, j}})\otimes K|_{U_x}\simeq
\bigoplus_i \EE_{U_{x, j} | U_x}^{\Re\varphi_i, \sub}.\]
Hence, for any $j$, we have isomorphisms in $\BEC(\I\CC_{U_x})$
$$\pi^{-1}(\iota_{U_x}\CC_{U_{x, j}})\otimes (I_X^\rmE K)|_{U_x}
\simeq
I_X^\rmE(\pi^{-1}(\rho_{U_x\ast}\CC_{U_{x, j}})\otimes K|_{U_x})
\simeq
I_X^\rmE\left(\bigoplus_i \EE_{U_{x, j} | U_x}^{\Re\varphi_i, \sub}\right)
\simeq
\bigoplus_i \EE_{U_{x, j} | U_x}^{\Re\varphi_i},$$
where in the first isomorphism we used \cite[Prop.\:3.7 (2)(iii),(vi), (4)(i)]{Ito21b}
and in the third isomorphism we used the fact that 
$I_X^\rmE(\EE_{U_{x, j} | U_x}^{\Re\varphi_i, \sub}) \simeq \EE_{U_{x, j} | U_x}^{\Re\varphi_i}$.
This implies that $I_X^\rmE K$ satisfies the condition (iii) in Definition \ref{def-normal}.
Therefore, we have $I_X^\rmE K\in \ZEC_{\rm norm}(\I\CC_{X(D)})$.

Let $K\in \ZEC_{\rm norm}(\I\CC_{X(D)})$.
Then we have an isomorphism
$\pi^{-1}(\iota_{X}\CC_{X\setminus D})\otimes K\simto K$ in $\BEC(\I\CC_X)$
and hence there exist isomorphisms in $\BEC(\CC_X^\sub)$
$$\pi^{-1}(\rho_{X\ast}\CC_{X\setminus D})\otimes \lambda_X^\rmE K
\simeq 
\lambda_X^\rmE(\pi^{-1}(\iota_{X}\CC_{X\setminus D})\otimes K)
\simeq
\lambda_X^\rmE K,$$
where in the second isomorphism we used \cite[Prop.\:3.7 (3)(i), (4)(ii),(iv)]{Ito21b}.
This implies that $\lambda_X^\rmE K$ satisfies the condition (i) in Definition \ref{def-normalsub}.
Let $x\in X\setminus D$.
Then there exists an open neighborhood $U_x\subset X\setminus D$
of $x$ and a non-negative integer $k$ such that
$K|_{U_x}\simeq (\CC_{U_x}^{\rmE})^{\oplus k}$
and hence we have isomorphism in $\BEC(\CC_{U_x}^\sub)$
$$(\lambda_X^\rmE K)|_{U_x}
\simeq
\lambda_{U_x}^\rmE(K|_{U_x})
\simeq
\lambda_{U_x}^\rmE((\CC_{U_x}^{\rmE})^{\oplus k})
\simeq
(\CC_{U_x}^{\rmE,\sub})^{\oplus k},$$
where in the first isomorphism we used \cite[Prop.\:3.19 (4)(i)]{Ito21b}
and the third isomorphism we used the fact that
$\lambda_{U_x}^\rmE(\CC_{U_x}^{\rmE}) \simeq \CC_{U_x}^{\rmE,\sub}$.
This implies that $\lambda_X^\rmE K$ satisfies the condition (ii) in Definition \ref{def-normalsub}.
Let $x\in D$. Then there exist an open neighborhood $U_x\subset X$ of $x$,
a good set of irregular values $\{\varphi_i\}_i$ on $(U_x, D\cap U_x)$
and a finite sectorial open covering $\{U_{x, j}\}_j$ of $U_x\bs D$
such that for any $j$
\[\pi^{-1}(\rho_{U_x\ast}\CC_{U_{x, j}})\otimes K|_{U_x}\simeq
\bigoplus_i \EE_{U_{x, j} | U_x}^{\Re\varphi_i}.\]
Hence, for any $j$, we have isomorphisms in $\BEC(\CC_{U_x}^\sub)$
$$\pi^{-1}(\rho_{U_x\ast}\CC_{U_{x, j}})\otimes (\lambda_X^\rmE K)|_{U_x}
\simeq
\lambda_X^\rmE(\pi^{-1}(\iota_{U_x}\CC_{U_{x, j}})\otimes K|_{U_x})
\simeq
\lambda_X^\rmE\left(\bigoplus_i \EE_{U_{x, j} | U_x}^{\Re\varphi_i}\right)
\simeq
\bigoplus_i \EE_{U_{x, j} | U_x}^{\Re\varphi_i, \sub},$$
where in the first isomorphism we used \cite[Prop.\:3.7 (3)(i), (4)(ii),(iv)]{Ito21b}
and in the third isomorphism we used the fact that 
$\lambda_X^\rmE(\EE_{U_{x, j} | U_x}^{\Re\varphi_i}) \simeq \EE_{U_{x, j} | U_x}^{\Re\varphi_i, \sub}$.
This implies that $\lambda_X^\rmE K$ satisfies the condition (iii) in Definition \ref{def-normalsub}.
Therefore, we have $\lambda_X^\rmE K\in \ZEC_{\rm norm}(\CC_{X(D)}^\sub)$.
\end{proof}

\begin{corollary}
The category $\ZEC_{\rm norm}(\CC_{X(D)}^\sub)$
is full subcategory of $\ZEC_{\RR\mbox{\scriptsize -}c}(\CC_{M_\infty}^\sub)$.
\end{corollary}

\begin{proof}
This follows from Propositions \ref{prop3.6}, \ref{prop3.8}
and the fact that $\ZEC_{\rm norm}(\CC_{X(D)}^\sub)$ is full subcategory of
$\ZEC_{\RR\mbox{\scriptsize -}c}(\I\CC_X)$.
\end{proof}

\begin{definition}\label{def-quasisub}
We say that an enhanced subanalytic sheaf
$K\in\ZEC(\CC_{X}^\sub)$ has a quasi-normal form along $D$
if it satisfies (i) and (ii) in Definition \ref{def-normalsub}, and if
for any $x\in D$ there exist an open neighborhood $U_x\subset X$ of $x$
and a ramification $r_x \colon U_x^{\rami}\to U_x$ of $U_x$ along $D_x := U_x\cap D$
such that $\bfE r_x^{-1}(K|_{U_x})$ has a normal form along $D_x^{\rami}:= r_x^{-1}(D_x)$.
\end{definition}

We shall denote by $\ZEC_{\rm q\mbox{\scriptsize -}norm}(\I\CC_{X(D)})$
the category of enhanced ind-sheaves which have quasi-normal form along $D$
and denote by $\ZEC_{\rm q\mbox{\scriptsize -}norm}(\CC_{X(D)}^\sub)$
the category of enhanced subanalytic sheaves which have quasi-normal form along $D$.
Remark that $\ZEC_{\rm q\mbox{\scriptsize -}norm}(\I\CC_{X(D)})$ is subcategory of
$\ZEC_{\RR\mbox{\scriptsize -}c}(\I\CC_X)$.
See \cite[Prop.\:3.12]{Ito20} for the details.

\begin{proposition}\label{prop3.11}
The functors $I_X^\rmE, \lambda_X^\rmE$ induce an equivalence of categories:
\[\xymatrix@M=7pt@C=45pt{
\ZEC_{\rm q\mbox{\scriptsize -}norm}(\CC_{X(D)}^\sub)\ar@<0.8ex>@{->}[r]^-{I_{X}^\rmE}_-\sim
&
\ZEC_{\rm q\mbox{\scriptsize -}norm}(\I\CC_{X(D)})
\ar@<0.8ex>@{->}[l]^-{\lambda_{X}^\rmE}.
}\]
\end{proposition}

\begin{proof}
By Proposition \ref{prop3.4},
it is enough to prove that 
$I_X^\rmE(\,\ZEC_{\rm q\mbox{\scriptsize -}norm}(\CC_{X(D)}^\sub)\,)
\subset \ZEC_{\rm q\mbox{\scriptsize -}norm}(\I\CC_{X(D)})$
and 
$\lambda_X^\rmE(\,\ZEC_{\rm q\mbox{\scriptsize -}norm}(\I\CC_{X(D)})\,)
\subset \ZEC_{\rm q\mbox{\scriptsize -}norm}(\CC_{X(D)}^\sub)$.

Let $K\in \ZEC_{\rm q\mbox{\scriptsize -}norm}(\CC_{X(D)}^\sub)$.
By the same argument as in the proof of Proposition \ref{prop3.8},
$I_X^\rmE K$ satisfies the conditions (i), (ii) in Definition \ref{def-normal}.
Let $x\in D$.
Then there exist an open neighborhood $U_x\subset X$ of $x$
and a ramification $r_x \colon U_x^{\rami}\to U_x$ of $U_x$ along $D_x := U_x\cap D$
such that $\bfE r_x^{-1}(K|_{U_x})$ has a normal form along $D_x^{\rami}:= r_x^{-1}(D_x)$.
Hence $I_{U_x^\rami}^\rmE(\bfE r_x^{-1}(K|_{U_x}))$ has a normal form along $D_x^{\rami}:= r_x^{-1}(D_x)$
by Proposition \ref{prop3.8}.
Moreover we have isomorphisms in $\BEC(\I\CC_{U_x^\rami})$
$$I_{U_x^\rami}^\rmE(\bfE r_x^{-1}(K|_{U_x}))
\simeq 
\bfE r_x^{-1}((I_X^\rmE K)|_{U_x})$$
by \cite[Prop.\:3.19 (2)(i)]{Ito21b}.
This implies that $I_X^\rmE K$ satisfies the conditions (iii) in Definition \ref{def-quasi}.
Therefore, we have $I_X^\rmE K\in \ZEC_{\rm q\mbox{\scriptsize -}norm}(\I\CC_{X(D)})$.

Let $K\in \ZEC_{\rm q\mbox{\scriptsize -}norm}(\I\CC_{X(D)})$.
By the same argument as in the proof of Proposition \ref{prop3.8},
$\lambda_X^\rmE K$ satisfies the conditions (i), (ii) in Definition \ref{def-normalsub}.
Let $x\in D$.
Then there exist an open neighborhood $U_x\subset X$ of $x$
and a ramification $r_x \colon U_x^{\rami}\to U_x$ of $U_x$ along $D_x := U_x\cap D$
such that $\bfE r_x^{-1}(K|_{U_x})$ has a normal form along $D_x^{\rami}:= r_x^{-1}(D_x)$.
Hence $\lambda_{U_x^\rami}^\rmE(\bfE r_x^{-1}(K|_{U_x}))$
has a normal form along $D_x^{\rami}:= r_x^{-1}(D_x)$
by Proposition \ref{prop3.8}.
Moreover we have isomorphisms in $\BEC(\CC_{U_x^\rami}^\sub)$
$$\lambda_{U_x^\rami}^\rmE(\bfE r_x^{-1}(K|_{U_x}))
\simeq 
\bfE r_x^{-1}((\lambda_X^\rmE K)|_{U_x})$$
by \cite[Prop.\:3.19 (4)(i)]{Ito21b}.
This implies that $\lambda_X^\rmE K$ satisfies the conditions (iii) in Definition \ref{def-quasisub}.
Therefore, we have $\lambda_X^\rmE K\in \ZEC_{\rm q\mbox{\scriptsize -}norm}(\CC_{X(D)}^\sub)$.
\end{proof}

\begin{corollary}
The category $\ZEC_{\rm q\mbox{\scriptsize -}norm}(\CC_{X(D)}^\sub)$
is full subcategory of $\ZEC_{\RR\mbox{\scriptsize -}c}(\CC_{M_\infty}^\sub)$.
\end{corollary}

\begin{proof}
This follows from Propositions \ref{prop3.6}, \ref{prop3.11} and \cite[Prop.\:3.12]{Ito20}.
\end{proof}

\begin{definition}\label{def-modisub}
We say that an enhanced subanalytic sheaf $K\in\ZEC(\CC_{X}^\sub)$
has a modified quasi-normal form along $H$
if it satisfies (i) and (ii) in Definition \ref{def-normalsub}, and if
for any $x\in H$ there exist an open neighborhood $U_x\subset X$ of $x$
and a modification $m_x \colon U_x^{\modi}\to U_x$ of $U_x$ along $H_x := U_x\cap H$
such that $\bfE m_x^{-1}(K|_{U_x})$ has a quasi-normal form along $D_x^{\modi} := m_x^{-1}(H_x)$.
\end{definition}

We shall denote by $\ZEC_{\rm modi}(\I\CC_{X(H)})$
the category of enhanced ind-sheaves which have modified quasi-normal form along $H$
and denote by $\ZEC_{\rm modi}(\CC_{X(H)}^\sub)$
the category of enhanced subanalytic sheaves which have modified quasi-normal form along $H$.
Remark that $\ZEC_{\rm modi}(\I\CC_{X(H)})$ is subcategory of
$\ZEC_{\RR\mbox{\scriptsize -}c}(\I\CC_X)$.
See \cite[Prop.\:3.15]{Ito20} for the details.

\begin{proposition}\label{prop3.14}
The functors $I_X^\rmE, \lambda_X^\rmE$ induce an equivalence of categories:
\[\xymatrix@M=7pt@C=45pt{
\ZEC_{\rm modi}(\CC_{X(H)}^\sub)\ar@<0.8ex>@{->}[r]^-{I_{X}^\rmE}_-\sim
&
\ZEC_{\rm modi}(\I\CC_{X(H)})
\ar@<0.8ex>@{->}[l]^-{\lambda_{X}^\rmE}.
}\]
\end{proposition}

\begin{proof}
By Proposition \ref{prop3.4},
it is enough to prove that 
$I_X^\rmE(\,\ZEC_{\rm modi}(\CC_{X(H)}^\sub)\,)
\subset \ZEC_{\rm modi}(\I\CC_{X(H)})$
and 
$\lambda_X^\rmE(\,\ZEC_{\rm modi}(\I\CC_{X(H)})\,)
\subset \ZEC_{\rm modi}(\CC_{X(H)}^\sub)$.

Let $K\in \ZEC_{\rm modi}(\CC_{X(H)}^\sub)$.
By the same argument as in the proof of Proposition \ref{prop3.8},
$I_X^\rmE K$ satisfies the conditions (i), (ii) in Definition \ref{def-normal}.
Let $x\in H$.
Then there exist an open neighborhood $U_x\subset X$ of $x$
and a modification $m_x \colon U_x^{\modi}\to U_x$ of $U_x$ along $H_x := U_x\cap H$
such that $\bfE m_x^{-1}(K|_{U_x})$ has a quasi-normal form along $D_x^{\modi} := m_x^{-1}(H_x)$.
Hence $I_{U_x^\modi}^\rmE(\bfE m_x^{-1}(K|_{U_x}))$ has a quasi-normal form
along $D_x^{\modi} := m_x^{-1}(H_x)$ by Proposition \ref{prop3.11}.
Moreover we have isomorphisms in $\BEC(\I\CC_{U_x^\modi})$
$$I_{U_x^\modi}^\rmE(\bfE m_x^{-1}(K|_{U_x}))
\simeq 
\bfE m_x^{-1}((I_X^\rmE K)|_{U_x})$$
by \cite[Prop.\:3.19 (2)(i)]{Ito21b}.
This implies that $I_X^\rmE K$ satisfies the conditions (iii) in Definition \ref{def-modi}.
Therefore, we have $I_X^\rmE K\in \ZEC_{\rm modi}(\I\CC_{X(H)})$.

Let $K\in \ZEC_{\rm modi}(\I\CC_{X(H)})$.
By the same argument as in the proof of Proposition \ref{prop3.8},
$\lambda_X^\rmE K$ satisfies the conditions (i), (ii) in Definition \ref{def-normal}.
Let $x\in H$.
Then there exist an open neighborhood $U_x\subset X$ of $x$
and a modification $m_x \colon U_x^{\modi}\to U_x$ of $U_x$ along $H_x := U_x\cap H$
such that $\bfE m_x^{-1}(K|_{U_x})$ has a quasi-normal form along $D_x^{\modi} := m_x^{-1}(H_x)$.
Hence $\lambda_{U_x^\modi}^\rmE(\bfE m_x^{-1}(K|_{U_x}))$ has a quasi-normal form
along $D_x^{\modi} := m_x^{-1}(H_x)$ by Proposition \ref{prop3.11}.
Moreover we have isomorphisms in $\BEC(\CC_{U_x^\modi}^\sub)$
$$\lambda_{U_x^\modi}^\rmE(\bfE m_x^{-1}(K|_{U_x}))
\simeq 
\bfE m_x^{-1}((\lambda_X^\rmE K)|_{U_x})$$
by \cite[Prop.\:3.19 (4)(i)]{Ito21b}.
This implies that $\lambda_X^\rmE K$ satisfies the conditions (iii) in Definition \ref{def-modisub}.
Therefore, we have $\lambda_X^\rmE K\in \ZEC_{\rm modi}(\CC_{X(H)}^\sub)$.
\end{proof}

\begin{corollary}
The category $\ZEC_{\rm modi}(\CC_{X(D)}^\sub)$
is full subcategory of $\ZEC_{\RR\mbox{\scriptsize -}c}(\CC_{M_\infty}^\sub)$.
\end{corollary}

\begin{proof}
This follows from Propositions \ref{prop3.6}, \ref{prop3.14} and \cite[Prop.\:3.15]{Ito20}.
\end{proof}

\begin{definition}
We say that an enhanced subanalytic sheaf $K\in\ZEC(\I\CC_{X})$ is $\CC$-constructible if
there exists a complex analytic stratification $\{X_\alpha\}_\alpha$ of $X$
such that
$$\pi^{-1}(\rho_{\var{X}^{\blow}_\alpha\ast}\CC_{\var{X}^{\blow}_\alpha\setminus D_\alpha})
\otimes \bfE b_\alpha^{-1}K$$
has a modified quasi-normal form along $D_\alpha$ for any $\alpha$.
Here $b_\alpha \colon \var{X}^{\blow}_\alpha \to X$ is a complex blow-up of $\var{X_\alpha}$
along $\partial X_\alpha = \var{X_\alpha}\setminus X_\alpha$ and
$D_\alpha := b_\alpha^{-1}(\partial X_\alpha)$.
\end{definition}

We denote by $\ZEC_{\CC\mbox{\scriptsize -}c}(\CC_{X}^\sub)$
the category of $\CC$-constructible enhanced subanalytic sheaves
and denote by $\ZEC_{\CC\mbox{\scriptsize -}c}(\I\CC_{X})$
the category of $\CC$-constructible enhanced ind-sheaves.
Remark that $\ZEC_{\CC\mbox{\scriptsize -}c}(\I\CC_X)$ is subcategory of
$\ZEC_{\RR\mbox{\scriptsize -}c}(\I\CC_X)$.
See \cite[Prop.\:3.21]{Ito20} for the details.

\begin{proposition}\label{prop3.17}
The functors $I_X^\rmE, \lambda_X^\rmE$ induce an equivalence of categories:
\[\xymatrix@M=7pt@C=45pt{
\ZEC_{\CC\mbox{\scriptsize -}c}(\CC_X^\sub)\ar@<0.8ex>@{->}[r]^-{I_{X}^\rmE}_-\sim
&
\ZEC_{\CC\mbox{\scriptsize -}c}(\I\CC_X)
\ar@<0.8ex>@{->}[l]^-{\lambda_{X}^\rmE}.
}\]
\end{proposition}

\begin{proof}
By Proposition \ref{prop3.4},
it is enough to prove that 
$I_X^\rmE(\,\ZEC_{\CC\mbox{\scriptsize -}c}(\CC_X^\sub)\,)
\subset \ZEC_{\CC\mbox{\scriptsize -}c}(\I\CC_X)$
and 
$\lambda_X^\rmE(\,\ZEC_{\CC\mbox{\scriptsize -}c}(\I\CC_X)\,)
\subset \ZEC_{\CC\mbox{\scriptsize -}c}(\CC_X^\sub)$.

Let $K\in\ZEC_{\CC\mbox{\scriptsize -}c}(\CC_X^\sub)$.
Then there exists a complex analytic stratification $\{X_\alpha\}_\alpha$ of $X$
such that for any $\alpha$
$$\pi^{-1}(\rho_{\var{X}^{\blow}_\alpha\ast}\CC_{\var{X}^{\blow}_\alpha\setminus D_\alpha})
\otimes \bfE b_\alpha^{-1}K$$
has a modified quasi-normal form along $D_\alpha$
and hence
$$I_{\var{X}^{\blow}_\alpha}^\rmE(
\pi^{-1}(\rho_{\var{X}^{\blow}_\alpha\ast}\CC_{\var{X}^{\blow}_\alpha\setminus D_\alpha})
\otimes \bfE b_\alpha^{-1}K)$$
has a modified quasi-normal form along $D_\alpha$ for any $\alpha$ by Proposition \ref{prop3.14}.
Moreover we have isomorphisms in $\BEC(\I\CC_{\var{X}^{\blow}_\alpha})$
$$I_{\var{X}^{\blow}_\alpha}^\rmE(
\pi^{-1}(\rho_{\var{X}^{\blow}_\alpha\ast}\CC_{\var{X}^{\blow}_\alpha\setminus D_\alpha})
\otimes \bfE b_\alpha^{-1}K)
\simeq
\pi^{-1}(\iota_{\var{X}^{\blow}_\alpha}\CC_{\var{X}^{\blow}_\alpha\setminus D_\alpha})
\otimes \bfE b_\alpha^{-1}(I_X^\rmE K)$$
by \cite[Prop.\:3.7 (2)(iii),(vi), (4)(i)]{Ito21b}.
Therefore, we have $I_X^\rmE K\in \ZEC_{\CC\mbox{\scriptsize -}c}(\I\CC_X)$.

Let $K\in\ZEC_{\CC\mbox{\scriptsize -}c}(\I\CC_X)$.
Then there exists a complex analytic stratification $\{X_\alpha\}_\alpha$ of $X$
such that for any $\alpha$
$$\pi^{-1}(\iota_{\var{X}^{\blow}_\alpha}\CC_{\var{X}^{\blow}_\alpha\setminus D_\alpha})
\otimes \bfE b_\alpha^{-1}K$$
has a modified quasi-normal form along $D_\alpha$
and hence
$$\lambda_{\var{X}^{\blow}_\alpha}^\rmE(
\pi^{-1}(\rho_{\var{X}^{\blow}_\alpha\ast}\CC_{\var{X}^{\blow}_\alpha\setminus D_\alpha})
\otimes \bfE b_\alpha^{-1}K)$$
has a modified quasi-normal form along $D_\alpha$ for any $\alpha$ by Proposition \ref{prop3.14}.
Moreover we have isomorphisms in $\BEC(\CC_{\var{X}^{\blow}_\alpha}^\sub)$
$$\lambda_{\var{X}^{\blow}_\alpha}^\rmE(
\pi^{-1}(\iota_{\var{X}^{\blow}_\alpha}\CC_{\var{X}^{\blow}_\alpha\setminus D_\alpha})
\otimes \bfE b_\alpha^{-1}K)
\simeq
\pi^{-1}(\rho_{\var{X}^{\blow}_\alpha\ast}\CC_{\var{X}^{\blow}_\alpha\setminus D_\alpha})
\otimes \bfE b_\alpha^{-1}(\lambda_X^\rmE K)$$
by \cite[Prop.\:3.7 (3)(i), (4)(ii),(iv)]{Ito21b}.
Therefore, we have $\lambda_X^\rmE K\in \ZEC_{\CC\mbox{\scriptsize -}c}(\CC_X^\sub)$.
\end{proof}

\begin{corollary}\label{cor3.18}
The category $\ZEC_{\CC\mbox{\scriptsize -}c}(\CC_X^\sub)$
is full subcategory of $\ZEC_{\RR\mbox{\scriptsize -}c}(\CC_X^\sub)$.
\end{corollary}

\begin{proof}
This follows from Propositions \ref{prop3.6}, \ref{prop3.17} and \cite[Prop.\:3.21]{Ito20}.
\end{proof}

We set
\[\BEC_{\CC\mbox{\scriptsize -}c}(\CC_{X}^\sub) := \{K\in\BEC(\I\CC_{X})\
|\ \SH^i(K)\in\ZEC_{\CC\mbox{\scriptsize -}c}(\CC_{X}^\sub) \mbox{ for any }i\in\ZZ \}\subset \BEC(\I\CC_{X}).\]

\begin{theorem}\label{main1}
The functors $I_X^\rmE, \lambda_X^\rmE$ induce an equivalence of triangulated categories:
\[\xymatrix@M=7pt@C=45pt{
\BEC_{\CC\mbox{\scriptsize -}c}(\CC_X^\sub)\ar@<0.8ex>@{->}[r]^-{I_X^\rmE}_-\sim
&
\BEC_{\CC\mbox{\scriptsize -}c}(\I\CC_X)
\ar@<0.8ex>@{->}[l]^-{\lambda_X^\rmE}.
}\]
\end{theorem}

\begin{proof}
This follows from Theorem \ref{thm2.6} and Proposition \ref{prop3.17}.
\end{proof}

\begin{corollary}
The category $\BEC_{\CC\mbox{\scriptsize -}c}(\CC_X^\sub)$
is full triangulated subcategory of $\BEC_{\RR\mbox{\scriptsize -}c}(\CC_X^\sub)$.
\end{corollary}

\begin{proof}
This follows from Corollary \ref{cor3.18}.
\end{proof}

\begin{proposition}\label{prop3.21}
The functor $e_X^\sub\circ \rho_{X\ast}$ induces an embedding:
$$e_X^\sub\circ \rho_{X\ast}\colon \BDC_{\CC-c}(\CC_X)\hookrightarrow\BEC_{\CC-c}(\CC_X^\sub).$$
\end{proposition}

\begin{proof}
This follows from Theorem \ref{main1}, \cite[Props.\:3.7 (3)(i), (4)(i), 3.24]{Ito21b} and \cite[Prop.\:3.1]{Ito23}.
\end{proof}

\begin{theorem}\label{main2}
%Let $X$ be a complex manifold.
There exists an equivalence of triangulated categories:
\[ \Sol_{X}^{\rmE, \sub} \colon \BDChol(\D_{X})^{\op}
\simto\BEC_{\CC\mbox{\scriptsize -}c}(\CC_{X}^\sub).\]
\end{theorem}

\begin{proof}
This follows from Theorems \ref{thm1.8}, \ref{thm-Ito20} and \ref{main1}.
\end{proof}

\begin{corollary}
We have a commutative diagram:
\[\xymatrix@C=30pt@M=5pt{
\BDChol(\D_X)^{\op}\ar@{->}[r]^\sim\ar@<1.0ex>@{}[r]^-{\Sol_X^{\rmE,\sub}}
\ar@{}[rd]|{\rotatebox[origin=c]{180}{$\circlearrowright$}}
 & \BEC_{\CC-c}(\CC_X^\sub)\\
\BDCrh(\D_X)^{\op}\ar@{->}[r]_-{\Sol_X}^-{\sim}\ar@{}[u]|-{\bigcup}
&\BDC_{\CC-c}(\CC_X).
\ar@{^{(}->}[u]_-{e_X^\sub\circ \rho_{X\ast}}
}\]
\end{corollary}

\begin{proof}
This follows from Proposition \ref{prop3.21},
\cite[Prop.\:3.24, Lem.\:3.40 (3)]{Ito21b} and 
the fact that $\Sol_X^\rmE(\M)\simeq e_X(\iota_X(\Sol_X(\M))$ for any $\M\in \BDCrh(\D_X)$.
See \cite[the equation just before Thm.\:9.1.2, Prop.\:9.1.3]{DK16}
(see also \cite[Last part of Prop.\:3.14]{Ito21}) for the details of the fact.
\end{proof}

One can summarize Theorem \ref{main2} and the results of \cite{Ito21b} in the following commutative diagram:
\[\xymatrix@M=5pt@R=45pt@C=45pt{
{}&{}&{}&\BDC(\CC_{X\times\RR_\infty}^\sub) & {}\\
\BDChol(\D_X)
\ar@{->}[r]_-{\Sol_X^{\rmE, \sub}}
  \ar@<-0.3ex>@{}[r]^-\sim
\ar@{^{(}->}[rrru]^-{\Sol_X^{\T, \sub}(\cdot)[1]}
\ar@{->}[rd]_-{\Sol_X^{\rmE}}
  \ar@<-1.0ex>@{}[rd]^-{\rotatebox{-25}{$\sim$}}
 & \BEC_{\CC\mbox{\scriptsize -}c}(\CC_X^\sub)
 \ar@{}[r]|-{\text{\large $\subset$}}
  \ar@<-1.0ex>@{->}[d]_-{I_X^\rmE}
 \ar@{}[d]|-\wr
 & \BEC_{\RR\mbox{\scriptsize -}c}(\CC_X^\sub)
 \ar@{}[r]|-{\text{\large $\subset$}}
 \ar@<-1.0ex>@{->}[d]_-{I_X^\rmE}
 \ar@{}[d]|-\wr
  & \BEC(\CC_X^\sub)
  \ar@{^{(}->}[u]_-{\bfR_X^{\rmE, \sub}}
  \ar@<-1.0ex>@{^{(}->}[rd]_-{I_X^\rmE}
   \ar@<-1.0ex>@{->}[d]_-{I_X^\rmE}
   \ar@{}[d]|-\wr\\
{}&\BEC_{\CC\mbox{\scriptsize -}c}(\I\CC_X)
\ar@{}[r]|-{\text{\large $\subset$}}
\ar@<-1.0ex>@{->}[u]_-{J_X^\rmE}
&\BEC_{\RR\mbox{\scriptsize -}c}(\I\CC_X)
\ar@{}[r]|-{\text{\large $\subset$}}
\ar@<-1.0ex>@{->}[u]_-{J_X^\rmE}
&\BEC_{\I\RR\mbox{\scriptsize -}c}(\I\CC_X)
\ar@{}[r]|-{\text{\large $\subset$}}
\ar@<-1.0ex>@{->}[u]_-{J_X^\rmE}
& \BEC(\I\CC_X).
\ar@<-1.0ex>@{->}[lu]_-{J_X^\rmE}
}\]
\bigskip

\subsection{Algebraic $\CC$-Constructible Enhanced Subanalytic Sheaves}
Throughout of this subsection, let $X$ be a smooth algebraic variety over $\CC$ and
denote by $X^\an$ the underlying complex manifold of $X$.

\begin{definition}
We say that an enhanced subanalytic sheaf $K\in\ZEC(\CC_{X^\an}^\sub)$ satisfies the condition 
$\ACsub$
if there exists an algebraic stratification $\{X_\alpha\}_\alpha$ of $X$ such that
$$\pi^{-1}(\rho_{(\var{X}^{\blow}_\alpha)^\an\ast}\CC_{(\var{X}^{\blow}_\alpha)^\an \setminus D^\an_\alpha})
\otimes \bfE (b^\an_\alpha)^{-1}K$$
has a modified quasi-normal form along $D^\an_\alpha$ for any $\alpha$,
Here $b_\alpha \colon \var{X}^\blow_\alpha \to X$ is a blow-up of $\var{X_\alpha}$
along $\partial X_\alpha := \var{X_\alpha}\setminus X_\alpha$,
$D_\alpha := b_\alpha^{-1}(\partial X_\alpha)$
and $D_\alpha^\an := \big(\var{X}_\alpha^\blow\big)^\an\setminus
\big(\var{X}_\alpha^\blow\setminus D_\alpha\big)^\an$.
\end{definition}

We denote by $\ZEC_{\CC\mbox{\scriptsize -}c}(\CC_X^\sub)$
the full subcategory of $\ZEC(\I\CC_{X^\an})$
whose objects satisfy the condition $\ACsub$
and 
denote by $\BEC_{\CC\mbox{\scriptsize -}c}(\CC_{X}^\sub)$
the full triangulated subcategory of $\BEC(\CC_{X}^\sub)$
consisting of objects whose cohomologies are contained in $\ZEC_{\CC\mbox{\scriptsize -}c}(\CC_{X}^\sub)$.
Remark that $\ZEC_{\CC\mbox{\scriptsize -}c}(\I\CC_X)$ is subcategory of
$\ZEC_{\RR\mbox{\scriptsize -}c}(\I\CC_{X^\an})$.
%See \cite[Prop.\:3.21]{Ito21} for the details.

\begin{proposition}\label{prop3.23}
The functors $I_{X^\an}^\rmE, \lambda_{X^\an}^\rmE$ induce an equivalence of categories:
\[\xymatrix@M=7pt@C=45pt{
\ZEC_{\CC\mbox{\scriptsize -}c}(\CC_X^\sub)\ar@<0.8ex>@{->}[r]^-{I_{X^\an}^\rmE}_-\sim
&
\ZEC_{\CC\mbox{\scriptsize -}c}(\I\CC_X)
\ar@<0.8ex>@{->}[l]^-{\lambda_{X^\an}^\rmE}.
}\]
Hence they induce an equivalence of triangulated categories:
\[\xymatrix@M=7pt@C=45pt{
\BEC_{\CC\mbox{\scriptsize -}c}(\CC_X^\sub)\ar@<0.8ex>@{->}[r]^-{I_{X^\an}^\rmE}_-\sim
&
\BEC_{\CC\mbox{\scriptsize -}c}(\I\CC_X)
\ar@<0.8ex>@{->}[l]^-{\lambda_{X^\an}^\rmE}.
}\]
\end{proposition}

\begin{proof}
By Theorem \ref{thm2.6} and Proposition \ref{prop3.4},
it is enough to show that 
$I_{X^\an}^\rmE(\,\ZEC_{\CC\mbox{\scriptsize -}c}(\CC_X^\sub)\,)
\subset \ZEC_{\CC\mbox{\scriptsize -}c}(\I\CC_X)$
and 
$\lambda_{X^\an}^\rmE(\,\ZEC_{\CC\mbox{\scriptsize -}c}(\I\CC_X)\,)
\subset \ZEC_{\CC\mbox{\scriptsize -}c}(\CC_X^\sub)$.

Let $K\in\ZEC_{\CC\mbox{\scriptsize -}c}(\CC_X^\sub)$.
Then there exists an algebraic stratification $\{X_\alpha\}_\alpha$ of $X$
such that  for any $\alpha$
$$\pi^{-1}(\rho_{(\var{X}^{\blow}_\alpha)^\an\ast}\CC_{(\var{X}^{\blow}_\alpha)^\an\setminus D_\alpha^\an})
\otimes \bfE(b_\alpha^\an)^{-1}K$$
has a modified quasi-normal form along $D_\alpha^\an$
and hence
$$I_{(\var{X}^{\blow}_\alpha)^\an}^\rmE(
\pi^{-1}(\rho_{(\var{X}^{\blow}_\alpha)^\an\ast}\CC_{(\var{X}^{\blow}_\alpha)^\an\setminus D_\alpha^\an})
\otimes \bfE(b_\alpha^\an)^{-1}K)$$
has a modified quasi-normal form along $D_\alpha^\an$ for any $\alpha$ by Proposition \ref{prop3.14}.
Moreover we have isomorphisms in $\BEC(\I\CC_{(\var{X}^{\blow}_\alpha)^\an})$
$$I_{(\var{X}^{\blow}_\alpha)^\an}^\rmE(
\pi^{-1}(\rho_{(\var{X}^{\blow}_\alpha)^\an\ast}\CC_{(\var{X}^{\blow}_\alpha)^\an\setminus D_\alpha^\an})
\otimes \bfE(b_\alpha^\an)^{-1}K)
\simeq
\pi^{-1}(\iota_{(\var{X}^{\blow}_\alpha)^\an}\CC_{(\var{X}^{\blow}_\alpha)^\an\setminus D_\alpha^\an})
\otimes \bfE(b_\alpha^\an)^{-1}(I_{X^\an}^\rmE K)$$
by \cite[Prop.\:3.7 (2)(iii),(vi), (4)(i)]{Ito21b}.
Therefore, we have $I_{X^\an}^\rmE K\in \ZEC_{\CC\mbox{\scriptsize -}c}(\I\CC_X)$.

Let $K\in\ZEC_{\CC\mbox{\scriptsize -}c}(\I\CC_X)$.
Then there exists an algebraic stratification $\{X_\alpha\}_\alpha$ of $X$
such that  for any $\alpha$
$$\pi^{-1}(\iota_{(\var{X}^{\blow}_\alpha)^\an}\CC_{(\var{X}^{\blow}_\alpha)^\an\setminus D_\alpha^\an})
\otimes \bfE(b_\alpha^\an)^{-1}K$$
has a modified quasi-normal form along $D_\alpha^\an$
and hence
$$\lambda_{(\var{X}^{\blow}_\alpha)^\an}^\rmE(
\pi^{-1}(\iota_{(\var{X}^{\blow}_\alpha)^\an}\CC_{(\var{X}^{\blow}_\alpha)^\an\setminus D_\alpha^\an})
\otimes \bfE(b_\alpha^\an)^{-1}K)$$
has a modified quasi-normal form along $D_\alpha^\an$ for any $\alpha$ by Proposition \ref{prop3.14}.
Moreover we have isomorphisms in $\BEC(\CC_{(\var{X}^{\blow}_\alpha)^\an}^\sub)$
$$\lambda_{(\var{X}^{\blow}_\alpha)^\an}^\rmE(
\pi^{-1}(\iota_{(\var{X}^{\blow}_\alpha)^\an}\CC_{(\var{X}^{\blow}_\alpha)^\an\setminus D_\alpha^\an})
\otimes \bfE(b_\alpha^\an)^{-1}K)
\simeq
\pi^{-1}(\rho_{(\var{X}^{\blow}_\alpha)^\an\ast}\CC_{(\var{X}^{\blow}_\alpha)^\an\setminus D_\alpha^\an})
\otimes \bfE(b_\alpha^\an)^{-1}(\lambda_{X^\an}^\rmE K)$$
by \cite[Prop.\:3.7 (3)(i), (4)(ii),(iv)]{Ito21b}.
Therefore, we have $\lambda_{X^\an}^\rmE K\in \ZEC_{\CC\mbox{\scriptsize -}c}(\CC_X^\sub)$.
\end{proof}

\begin{corollary}
The category $\ZEC_{\CC\mbox{\scriptsize -}c}(\CC_X^\sub)$
is full subcategory of $\ZEC_{\CC\mbox{\scriptsize -}c}(\CC_{X^\an}^\sub)$.
Hence the category $\BEC_{\CC\mbox{\scriptsize -}c}(\CC_X^\sub)$
is full triangulated subcategory of $\BEC_{\CC\mbox{\scriptsize -}c}(\CC_{X^\an}^\sub)$.
\end{corollary}

\begin{proof}
This follows from Propositions \ref{prop3.17}, \ref{prop3.23} and \cite[Prop.\:3.3]{Ito21}.
\end{proof}

\begin{theorem}
Let $X$ be a smooth complete algebraic variety over $\CC$.
Then there exists an equivalence of triangulated categories:
\[\Sol_X^{\rmE,\sub} \colon \BDChol(\D_X)^{\op}\simto \BEC_{\CC\mbox{\scriptsize -}c}(\I\CC_X),\
\M\mapsto \Sol_X^{\rmE,\sub}(\M) := \Sol_{X^\an}^{\rmE, \sub}(\M^\an).\]
\end{theorem}

\begin{proof}
This follows from Proposition \ref{prop3.23}, \cite[Thm.\:3.7]{Ito21} and \cite[Lem.\:3.40 (3)]{Ito21b}.
\end{proof}

Let $\tl{X}$ be a smooth complete algebraic variety such that $X\subset \tl{X}$
and $D := \tl{X}\setminus X$ is a normal crossing divisor of $\tl{X}$.
We shall consider a bordered space $X^\an_\infty = (X^\an, \tl{X}^\an)$
and the triangulated category $\BEC(\CC_{X^\an_\infty}^\sub)$ of enhanced subanalytic sheaves 
on $X^\an_\infty$.
Note that $\BEC(\CC_{X^\an_\infty}^\sub)$ does not depend on the choice of $\tl{X}$.

\begin{definition}
We say that an enhanced subanalytic sheaf $K\in\BEC(\CC_{X^\an_\infty}^\sub)$ is
algebraic $\CC$-constructible on $X_\infty^\an$
if $\bfE j_{X_\infty^\an!!}K \in\BEC_{\CC\mbox{\scriptsize -}c}(\CC_{\tl{X}}^\sub)$.
\end{definition}

We denote by $\BEC_{\CC\mbox{\scriptsize -}c}(\CC_{X_\infty}^\sub)$
the full triangulated subcategory of $\BEC(\CC_{X^\an_\infty}^\sub)$
consisting of algebraic $\CC$-constructible enhanced subanalytic sheaves on $X_\infty^\an$.
Remark that $\ZEC_{\CC\mbox{\scriptsize -}c}(\I\CC_{X_\infty})$ is subcategory of
$\ZEC_{\RR\mbox{\scriptsize -}c}(\I\CC_{X_\infty^\an})$.
%See \cite[Prop.\:3.21]{Ito21} for the details.

\begin{theorem}\label{main3}
The functors $I_{X^\an}^\rmE, \lambda_{X^\an}^\rmE$ induce an equivalence of triangulated categories:
\[\xymatrix@M=7pt@C=45pt{
\BEC_{\CC\mbox{\scriptsize -}c}(\CC_{X_\infty}^\sub)\ar@<0.8ex>@{->}[r]^-{I_{X_\infty^\an}^\rmE}_-\sim
&
\BEC_{\CC\mbox{\scriptsize -}c}(\I\CC_{X_\infty})
\ar@<0.8ex>@{->}[l]^-{\lambda_{X_\infty^\an}^\rmE}.
}\]
\end{theorem}

\begin{proof}
By Proposition \ref{prop3.4},
it is enough to prove that 
$I_{X_\infty^\an}^\rmE(\,\BEC_{\CC\mbox{\scriptsize -}c}(\CC_{X_\infty}^\sub)\,)
\subset \BEC_{\CC\mbox{\scriptsize -}c}(\I\CC_{X_\infty})$
and 
$\lambda_X^\rmE(\,\BEC_{\CC\mbox{\scriptsize -}c}(\I\CC_{X_\infty})\,)
\subset \BEC_{\CC\mbox{\scriptsize -}c}(\CC_{X_\infty}^\sub)$.

Let $K\in \BEC_{\CC\mbox{\scriptsize -}c}(\CC_{X_\infty}^\sub)$.
Then we have $\bfE j_{X_\infty^\an!!}K\in \BEC_{\CC\mbox{\scriptsize -}c}(\CC_{\tl{X}}^\sub)$
and hence $I_{\tl{X}^\an}^\rmE(\bfE j_{X_\infty^\an!!}K)\in \BEC_{\CC\mbox{\scriptsize -}c}(\I\CC_{\tl{X}})$
by Proposition \ref{prop3.23}.
Moreover we have isomorphisms in $\BEC(\I\CC_{\tl{X}})$
$$I_{\tl{X}^\an}^\rmE(\bfE j_{X_\infty^\an!!}K)\simeq \bfE j_{X_\infty^\an!!}(I_{X_\infty^\an}^\rmE K)$$
by \cite[Prop.\:3.19 (2)(ii)]{Ito21b}.
This implies that $I_{X_\infty^\an}^\rmE K\in \BEC_{\CC\mbox{\scriptsize -}c}(\I\CC_{X_\infty})$.

Let $K\in \BEC_{\CC\mbox{\scriptsize -}c}(\I\CC_{X_\infty})$.
Then we have $\bfE j_{X_\infty^\an!!}K\in \BEC_{\CC\mbox{\scriptsize -}c}(\I\CC_{\tl{X}})$
and hence $\lambda_{\tl{X}^\an}^\rmE(\bfE j_{X_\infty^\an!!}K)\in
\BEC_{\CC\mbox{\scriptsize -}c}(\CC_{\tl{X}}^\sub)$
by Proposition \ref{prop3.23}.
Moreover we have isomorphisms in $\BEC(\CC_{\tl{X}}^\sub)$
$$\lambda_{\tl{X}^\an}^\rmE(\bfE j_{X_\infty^\an!!}K)\simeq
\bfE j_{X_\infty^\an!!}(\lambda_{X_\infty^\an}^\rmE K)$$
by \cite[Prop.\:3.19 (4)(ii)]{Ito21b}.
This implies that $\lambda_{X_\infty^\an}^\rmE K\in \BEC_{\CC\mbox{\scriptsize -}c}(\CC_{X_\infty}^\sub)$.
\end{proof}

\begin{corollary}
The category $\BEC_{\CC\mbox{\scriptsize -}c}(\CC_{X_\infty}^\sub)$
is full triangulated subcategory of $\BEC_{\RR\mbox{\scriptsize -}c}(\CC_{X_\infty^\an}^\sub)$.
\end{corollary}

\begin{proof}
This follows from Proposition \ref{prop3.6}, Theorem \ref{main3}
and the fact that
the triangulated category $\BEC_{\CC\mbox{\scriptsize -}c}(\I\CC_{X_\infty})$ is the full triangulated subcategory
of $\BEC_{\RR\mbox{\scriptsize -}c}(\I\CC_{X^\an_\infty})$.
\end{proof}

\begin{proposition}\label{prop3.31}
The functor $e_{X_\infty}^\sub\circ \rho_{X_\infty\ast}$ induces an embedding
$$e_{X_\infty}^\sub\circ \rho_{X_\infty\ast}\colon
\BDC_{\CC-c}(\CC_X)\hookrightarrow\BEC_{\CC-c}(\CC_{X_\infty}^\sub).$$
\end{proposition}

\begin{proof}
This follows from Theorem \ref{main3}, \cite[Props.\:3.7 (3)(i), (4)(i), 3.24]{Ito21b} and \cite[Prop.\:3.7]{Ito23}.
\end{proof}

Let us consider a functor
$$\Sol_{X_\infty}^{\rmE,\sub} \colon
\BDC(\D_X)^{\op}\to \BEC(\CC_{X_\infty}^\sub),\
\M\mapsto \bfE j_{X_\infty^\an}^{-1}\Sol_{\tl{X}}^{\rmE, \sub}(\bfD j_\ast\M).$$
Here $\bfD j_!$ is the proper direct image functor for algebraic $\D$-modules
by the open embedding $j\colon X\to\tl{X}$.

\begin{lemma}\label{lem3.32}
For any $\M\in\BDChol(\D_X)$,
there exists an isomorphism in $\BEC(\CC_{X_\infty^\an}^\sub):$
$$\Sol_{X_\infty}^{\rmE, \sub}(\M)\simeq
J_{X_\infty^\an}^\rmE\Sol_{X_\infty}^{\rmE}(\M).$$
\end{lemma}

\begin{proof}
This follows from the fact that $\bfE j_{X_\infty^\an}^!\simeq \bfE j_{X_\infty^\an}^{-1}$
and \cite[Prop.\:3.19 (3)(ii), Lem.\:3.40 (3)]{Ito21b}.
\end{proof}

\begin{theorem}\label{main4}
The functor $\Sol_{X_\infty}^{\rmE,\sub}$ induces an equivalence of triangulated categories:
$$\Sol_{X_\infty}^{\rmE,\sub} \colon
\BDChol(\D_X)^{\op}\simto \BEC_{\CC\mbox{\scriptsize -}c}(\CC_{X_\infty}^\sub),\
\M\mapsto \bfE j_{X_\infty^\an}^{-1}\Sol_{\tl{X}}^{\rmE, \sub}(\bfD j_\ast\M)$$
and the following diagrams are commutative:
\[\xymatrix@M=7pt@R=45pt@C=70pt{
\BDChol(\D_X)
\ar@{->}[r]_-{\Sol_{X_\infty}^{\rmE, \sub}}^-\sim
\ar@{->}[rd]_-{\Sol_{X_\infty}^{\rmE}}^-\sim
 & \BEC_{\CC\mbox{\scriptsize -}c}(\CC_X^\sub)
 \ar@<-1.0ex>@{->}[d]_-{I_{X_\infty^\an}^\rmE}\ar@{}[d]|-\wr\\
{}&\BEC_{\CC\mbox{\scriptsize -}c}(\I\CC_X),}
\hspace{37pt}
\xymatrix@M=7pt@R=45pt@C=70pt{
\BDChol(\D_X)
\ar@{->}[r]_-{\Sol_{X_\infty}^{\rmE, \sub}}^-\sim
\ar@{->}[rd]_-{\Sol_{X_\infty}^{\rmE}}^-\sim
 & \BEC_{\CC\mbox{\scriptsize -}c}(\CC_X^\sub)\\
{}&\BEC_{\CC\mbox{\scriptsize -}c}(\I\CC_X)\ar@<-1.0ex>@{->}[u]_-{\lambda_{X_\infty^\an}^\rmE}\ar@{}[u]|-\wr.
}\]
\end{theorem}

\begin{proof}
The first assertion follows from Lemma \ref{lem3.32}, Theorem \ref{main3} and \cite[Thm.\:3.11]{Ito21}.
The second  assertion follows from Lemma \ref{lem3.32}, \cite[Thm.\:3.11]{Ito21} and Theorem \ref{main3}.
The third assertion follows from Lemma \ref{lem3.32} and Theorem \ref{main3}.
\end{proof}

\begin{corollary}
We have a commutative diagram:
\[\xymatrix@C=30pt@M=5pt{
\BDChol(\D_X)^{\op}\ar@{->}[r]^\sim\ar@<1.0ex>@{}[r]^-{\Sol_{X_\infty}^{\rmE,\sub}}
\ar@{}[rd]|{\rotatebox[origin=c]{180}{$\circlearrowright$}}
 & \BEC_{\CC-c}(\CC_{X_\infty}^\sub)\\
\BDCrh(\D_X)^{\op}\ar@{->}[r]_-{\Sol_X}^-{\sim}\ar@{}[u]|-{\bigcup}
&\BDC_{\CC-c}(\CC_X).
\ar@{^{(}->}[u]_-{e_{X_\infty}^\sub\circ \rho_{X_\infty\ast}}
}\]
\end{corollary}

\begin{proof}
This follows from Lemma \ref{lem3.32} and \cite[Props.\:3.7 (3)(i), (4)(i), 3.24]{Ito21b}.
\end{proof}

\end{document}